\documentclass{article}
\usepackage{amsmath, amsthm, amssymb}
\usepackage{enumitem}
\usepackage{hyperref}

\newtheorem{theorem}{Theorem}[section]
\newtheorem{definition}[theorem]{Definition}
\newtheorem{proposition}[theorem]{Proposition}
\newtheorem{lemma}[theorem]{Lemma}
\newtheorem{example}[theorem]{Example}
\newtheorem{remark}[theorem]{Remark}

\DeclareUnicodeCharacter{220E}{\ensuremath{\square}}

\newcommand{\Prov}{\mathsf{Prov}}

\newcommand{\Proof}{\mathsf{Proof}}

\title{Nonlinear Continuum of States and Intuitionistic Flows in a Cognitive Space}
\author{
  Faruk Alpay\footnote{Lightcap, Department of Future, \texttt{alpay@lightcap.ai}}\and
  Taylan Alpay\footnote{Aerospace Engineering, Turkish Aeronautical Association, \texttt{s220112602@stu.thk.edu.tr}}
}

\begin{document}
\maketitle

% To reduce overfull \hbox warnings by allowing more flexible line breaking
\sloppy

\begin{abstract}
We construct a rigorous mathematical framework for an abstract continuous evolution of an internal state, inspired by the intuitive notion of a flowing thought sequence. Using tools from topology, functional analysis, measure theory, and logic, we formalize an indefinitely proceeding sequence of states as a non‑linear continuum with rich structure. In our development, a trajectory of states is modeled as a continuous mapping on a topological state space (a potentially infinite‑dimensional Banach space) and is further conceptualized intuitionistically as a choice sequence not fixed in advance. We establish fundamental properties of these state flows, including existence and uniqueness of evolutions under certain continuity conditions (via a Banach fixed‑point argument), non‑measurability results (demonstrating the impossibility of assigning a classical measure to all subsets of the continuum of states), and logical semantic frameworks (defining a Tarskian truth definition for propositions about states). Throughout, we draw on ideas of Brouwer, Banach, Tarski, Poincaré, Hadamard, and others—blending intuitionistic perspective with classical analysis—to rigorously capture a continuous, ever‑evolving process of an abstract cognitive state without resorting to category‑theoretic notions. This manuscript is presented in formal \LaTeX{}, with definitions, lemmas, propositions, and proofs, as a self‑contained study of a mathematical model of a non‑linear stream of an internal state.
\end{abstract}

\section{Introduction}

In mathematical terms, one may view the evolving state of a thinking subject as a path in an appropriate state space. Our aim is to formalize this intuitive continuum of internal states using the language of pure mathematics. We consider an abstract space $M$ (the state space) whose elements represent possible configurations of this internal state at an instant. A priori, $M$ can be quite general—possibly an infinite‑dimensional space or even a class in some set‑theoretic universe (in the sense of Grothendieck) to accommodate a rich collection of states. The evolution of the state will be described by a family of mappings or a flow on $M$.

We seek a framework that captures the continuous, seemingly unbroken progression of these states, while allowing for complex, non‑linear transitions. Classical real time will index the evolution, but we will also investigate an intuitionistic treatment wherein time and the state sequence are constructed incrementally, never as a completed whole. Our development will therefore proceed on two parallel tracks:
\begin{itemize}
  \item \emph{Topological and analytical track}: We treat $M$ as a topological (and often metric) space, possibly a Banach space, and define continuous trajectories $x:\mathbb{R}\to M$ that represent the flow of states over (real) time. Within this track, we apply tools such as the Banach fixed‑point theorem to guarantee existence and uniqueness of certain trajectories, and explore measure‑theoretic properties (or obstructions) of these flows.
  \item \emph{Intuitionistic and logical track}: We also formulate the evolution as a choice sequence (in the sense of Brouwer) of states, emphasizing that the sequence is not predetermined by any finished law, embodying the idea of a free progression of thought. We connect this with a formal semantic perspective: using a logical language to describe properties of states and employing Tarski’s semantic theory of truth to interpret statements about the evolving state.
\end{itemize}

Although this manuscript does not claim to supplant existing unified frameworks for modelling cognitive or dynamical processes, a brief comparison is instructive.  Some recent approaches, such as categorical or information‑geometric theories of cognition, aim to provide a single mathematical formalism encompassing diverse kinds of dynamics (for instance, discrete and continuous updates, or interactions across multiple scales).  These often employ tools like sheaf theory, higher‑order category theory or information geometry.  By contrast, our framework deliberately synthesises classical topological analysis, measure theory and intuitionistic logic without appealing to such higher abstractions.  Rather than unifying different paradigms, we sought to formalise an abstract notion of a continuously unfolding state within a well‑known analytical setting.  A key limitation of our earlier exposition was the absence of a worked instantiation illustrating how the general definitions operate in practice; this has been remedied by the inclusion of Example~\ref{ex:circle-rotation}, which demonstrates how recurrence and ergodicity manifest in a concrete circle rotation.  We also stress that assumptions like the existence of a finite invariant measure, crucial in Theorem~\ref{thm:recurrence}, are not consequences of the basic setup but must be imposed externally, much as in classical ergodic theory.

Finally, we recognise that the history of mathematical thought spans many cultures.  A notable example comes from Chinese mathematics: Liu~Hui’s third‑century commentary on the \emph{Nine Chapters on the Mathematical Art} not only provided recipes for solving practical problems but, more importantly, advanced a more mathematical mode of explanation.  His commentary supplied underlying principles for the rules, investigated the accuracy of the approximations, and even gave early evidence of ideas connected with differential and integral calculus.  Such a shift from mere prescriptions to explanatory analysis prefigures the rigorous approach we adopt here.

In the twentieth century, Chinese mathematics produced further major contributions to the global discipline.  Hua~Luogeng (Hua~Loo–Keng) became one of the leading mathematicians of his era and a central figure in modern Chinese mathematics.  His papers on number theory—especially on Waring’s problem and exponential sums—are regarded as an index of the subject’s major activities in the first half of the twentieth century.  Hua’s instinct for important problems, his powerful techniques, and his leadership over five decades helped to cultivate a thriving mathematical community in China.  These and other cross‑cultural contributions emphasise that the methods deployed throughout this paper are part of a long tradition of analytical thought spanning many times and places.

By integrating these approaches, we obtain a precise but abstract model of a continuously evolving state that reflects the intuitive notion of a self‑unfolding, perhaps non‑measurable, stream of mental content.  Our aim is not to replace categorical unification schemes but to provide an alternative viewpoint that does not invoke higher‑categorical machinery.  Instead we work with set‑theoretic, topological, analytical and logical formalisms, albeit guided by the structural insights championed by Grothendieck and Mac Lane.

\subsection*{Related influences}

The general idea of an internal state continuum has philosophical roots: Brouwer’s intuitionism considers the continuum (the ``intuitionistic continuum'') as a primitive given through the ``two‑oneness'' of time and the creative act of the mind. Indeed, Brouwer introduced the notion of a Creating Subject, an idealized mind constructing mathematical objects in time, particularly via choice sequences. This perspective inspires our intuitionistic formulation of the state trajectory. On the analytical side, Poincaré’s work on dynamical systems and recurrences hints at how a deterministic system can exhibit perpetual novelty or eventual repetition; for instance, the Poincaré recurrence theorem guarantees that certain flows return arbitrarily close to past states. Banach’s contributions to functional analysis provide the fixed‑point principle that undergirds the existence of solutions to evolution equations, and Tarski’s work in logic provides tools to discuss the truth of statements about states in a rigorous way. Additionally, the surprising result of Banach and Tarski on the existence of non‑measurable sets serves as a caution when attempting to assign a measure or ``volume'' to parts of the continuum of states—suggesting that a classical measure theory may not fully capture the ``size'' of fragments of an internal continuum if we assume maximal set‑theoretic freedom (Choice). Finally, we note Hadamard’s early study of chaotic geodesic flows on surfaces of negative curvature, illustrating that even simple deterministic rules can produce highly intricate, non‑repeating trajectories; this informs our view that the evolution of states might be highly non‑linear or sensitive to initial conditions, despite being continuous.

The paper is organized as follows. Section~\ref{sec:state} introduces the state space $M$ and its basic structures (topology, metric, measure) and defines what we mean by a continuous state trajectory or flow, including a fixed‑point existence result (Proposition~\ref{prop:banach}). Section~\ref{sec:intuitionistic} develops an intuitionistic perspective, formalizing the state continuum as a choice sequence and reconciling this with the classical model. Section~\ref{sec:logic} builds a logical semantic framework over the evolving state. We also examine measure‑theoretic aspects: Proposition~\ref{prop:nonmeas} shows that, under the axiom of choice, non‑measurable subsets of the state continuum exist. Section~\ref{sec:dynamics} studies dynamical consequences, including recurrence, stability, and chaos. We conclude with a reflection on the scope and limits of this formalism.
Throughout, all mathematical arguments are given with full rigor. We strive to keep the presentation self‑contained; definitions and preliminary results are provided as needed. The style is that of pure mathematics – our intent is that a reader versed in topology, analysis, and logic can appreciate the results without any prior knowledge of the motivating psychological concept, yet those familiar with the concept will recognize its formal reflection in the theorems and constructions that follow.

\section{State space and continuous flows}\label{sec:state}

% Use texorpdfstring to avoid math in bookmarks
\subsection{The state space \texorpdfstring{$M$}{M}}

Let $M$ be a nonempty set. We think of $M$ abstractly as the ``space of all possible internal states'' of our subject, but mathematically $M$ will be treated as a set equipped with additional structure. Specifically, we assume:
\begin{itemize}
  \item \emph{Topological structure}: $M$ is endowed with a topology $\tau$, making $(M,\tau)$ a topological space. We do not assume $M$ is necessarily metrizable or Hausdorff at the start, to allow generality. However, in many examples one may consider $M$ to be a metric space (even a Polish space or a Banach space under some metric $d$) for convenience. When needed, we will assume metric structure. For most of the general theory, $M$ need only be a topological space satisfying appropriate separation axioms so that usual theorems apply.
  \item \emph{Algebraic or linear structure (optional)}: In some instances, we will treat $M$ as a vector space over $\mathbb{R}$ (or $\mathbb{C}$), and in particular as a normed space (Banach space when complete). This will allow us to use analytic tools like differentiation of state trajectories and application of Banach’s fixed‑point theorem.
  \item \emph{Measure structure (optional)}: We may consider a $\sigma$‑algebra $\mathcal{F}$ of subsets of $M$ and a measure $\mu:\mathcal{F}\to [0,+\infty]$ to discuss measurable events or quantities related to states. However, we will show that if we try to take $\mathcal{F}$ to be the power set $2^M$, serious obstructions arise (in fact, under the axiom of choice, typically no finitely additive measure can assign reasonable ``volumes'' to all subsets of a rich continuum).
\end{itemize}

We will often refer to elements of $M$ simply as \emph{states}. No further structure is assumed on $M$ a priori, though later we will introduce additional axioms or properties (like compactness, connectedness, completeness, etc.) as needed for specific results.

For set‑theoretic safety (to avoid Russell‑style paradoxes when considering ``the set of all sets'' or similar large collections), we mention without detail that one can work inside a Grothendieck universe $U$ that contains $M$ and is closed under the usual set operations. In practice, this means any power set or product we form from elements of $U$ is also an element of $U$. We shall not dwell on this set‑theoretic foundation, but it allows us, for example, to consider the set $M^{\mathbb{R}}$ of all functions from $\mathbb{R}$ to $M$ as a legitimate set (element of a larger universe) rather than a proper class.

\subsection{Continuous trajectories}

We formalize the evolution of the state as a mapping $x:T\to M$ where $T$ is a time index set. We will usually take $T=\mathbb{R}$ (interpreted as physical time, which we assume to be continuous), or a subset like $[0,1]$ or $\mathbb{R}_{\ge 0}$ for an initial value problem.

\begin{definition}[Trajectory]\label{def:trajectory}
A \emph{state trajectory} (or \emph{flow path}) is a function $x:\mathbb{R}\to M$ (with $t\mapsto x(t)$) which is assumed to be continuous when $\mathbb{R}$ is given its standard topology and $M$ has the topology $\tau$. Equivalently, for every open set $U\in \tau$, the preimage $x^{-1}(U)\subset \mathbb{R}$ is open in $\mathbb{R}$. If in addition $M$ is equipped with a metric $d$, we further require that $x:\mathbb{R}\to (M,d)$ is a continuous (and often differentiable) map in the metric sense.
\end{definition}

We sometimes call such an $x(t)$ a \emph{flow}, thinking of $x$ as describing the motion of a point in the space $M$. In classical dynamical systems terms, one might have a family of maps $\{\Phi^t:M\to M\}_{t\in\mathbb{R}}$ forming a continuous flow (i.e.\ $\Phi^0=\mathrm{Id}_M$, $\Phi^{t+s}=\Phi^t\circ \Phi^s$). In that case, $x(t)=\Phi^t(x(0))$ is a trajectory obeying the semi‑group property. We do not assume the existence of a global flow $\Phi^t$ a priori; instead, we often construct trajectories directly via differential or integral equations.

\begin{example}[Trivial trajectory]\label{ex:constant}
A constant map $x(t)=s_0$ for all $t$, where $s_0\in M$ is fixed, is a valid trajectory. This represents a steady state or unchanging state over time. While trivial, such constant solutions will play a role as equilibrium solutions when we study dynamical equations. Non‑trivial trajectories, of course, are of greater interest as they model changing states.
\end{example}
\begin{remark}
In the recurrence theorem and the preceding example we worked with integer iterates of the flow.  This is sufficient to obtain recurrence for almost every initial point, because the set of integer times is unbounded and the measure is invariant under the flow.  If one wishes to obtain recurrence along the full continuum of times, one can use the fact that the flow $\Phi^t$ is continuous in $t$ and that the set of return times is dense: once a return occurs at some integer time $n$, nearby real times yield states arbitrarily close to the starting point.  Classical arguments in ergodic theory (using Birkhoff’s pointwise ergodic theorem and conservativity) show that for measure‑preserving, conservative flows, recurrence along real times holds for almost every point.  We do not pursue these refinements here, as the integer‑time version suffices to illustrate the phenomenon.
\end{remark}
\begin{example}[Dissipative logistic map on a compact attractor]\label{ex:dissipative}
As a second worked example illustrating sensitive dependence on initial conditions, consider the discrete‑time dynamical system defined by the logistic map.  Let $M=[0,1]$ and define $T:M\to M$ by
\[
  T(x) = 4x(1-x).
\]
  The interval $[0,1]$ is forward invariant under $T$ and serves as a compact attractor: iterates of $T$ push almost all initial points into a chaotic invariant subset of $[0,1]$.  The map $T$ is continuous but not invertible, and it expands distances on average.  In fact, given any $\delta>0$ there exist points $x,y\in[0,1]$ arbitrarily close (with $|x-y|<\delta$) and an iterate $n$ such that $|T^n(x)-T^n(y)|>\tfrac{1}{2}$.  This is one manifestation of \emph{sensitive dependence on initial conditions}.

  Beyond this qualitative description one can give a precise statistical picture.  There exists an absolutely continuous invariant probability measure (acim) $\mu$ for $T$ whose density is 
  \[ h(x)=\tfrac{1}{\pi\sqrt{x(1-x)}} \]
  on $(0,1)$.  One can verify this formula by analysing the Perron–Frobenius transfer operator 
  \[ [\mathcal{L}\varphi](x) = \sum_{y:T(y)=x} \frac{\varphi(y)}{|T'(y)|}. \]
  For the logistic map $T(x)=4x(1-x)$ each $x\in(0,1)$ has two preimages $y_{\pm}(x)=\tfrac{1\pm\sqrt{1-x}}{2}$, and a direct computation shows that $h$ satisfies the fixed‑point equation $\mathcal{L}h = h$: one checks that 
  $h(y_{+}(x)) + h(y_{-}(x)) = 2 h(x)$ and that $|T'(y_{\pm}(x))| = 2\sqrt{1-x}$, so the Jacobian factors cancel and $[\mathcal{L}h](x)=h(x)$.  By standard Perron–Frobenius theory (see Chapter~IV of de~Melo and van~Strien’s “One‑Dimensional Dynamics”) this implies that $h$ is an invariant density for $T$, and a spectral gap argument shows that it is unique among integrable densities.

  With respect to this acim the map is not only ergodic but also mixing: if $\varphi,\psi:[0,1]\to \mathbb{R}$ are Hölder continuous functions (with respect to the usual Euclidean metric) of exponent $\beta\in(0,1]$, then there exist constants $C>0$ and $\rho\in(0,1)$ such that the correlations
  \[
    \left|\int \varphi\circ T^n\cdot \psi\,d\mu \; - \; \int \varphi\,d\mu\int \psi\,d\mu\right| \le C\,\rho^n\,\|\varphi\|_{C^{\beta}}\,\|\psi\|_{C^{\beta}}
  \]
  decay exponentially in $n$.  Here $\|\varphi\|_{C^{\beta}}$ denotes the Hölder norm
  \[
    \|\varphi\|_{C^{\beta}} = \sup_{x\neq y}\frac{|\varphi(x)-\varphi(y)|}{|x-y|^{\beta}} + \sup_{x} |\varphi(x)|.
  \]
  These spectral estimates for the transfer operator provide a rigorous justification for the claim that correlations of Hölder observables decay exponentially fast.

  Thus the logistic map supplies a non‑trivial, dissipative example within our framework.  It contrasts with the conservative rotation of Example~\ref{ex:circle-rotation} and provides a concrete setting for the chaotic behaviour discussed in Proposition~\ref{prop:chaos}.  The presence of an acim and mixing ensures that time averages converge to space averages for a broad class of observables, thereby illustrating how measure‑theoretic phenomena can be accommodated within our topological–analytical model.
  \paragraph{Coordinate calculation.}  To illustrate explicitly that the gradient flow of a convex divergence does not depend on the choice of \(\alpha\)–connection, write a point in the simplex as $p=(p_1,\dots,p_n)$ with $p_i>0$ and $\sum_{i=1}^n p_i=1$.  The Fisher metric in these affine coordinates is $g_{ij}(p)=\delta_{ij}/p_i$.  For the Kullback--Leibler divergence
  \[
    D(p\Vert q)=\sum_{i=1}^n p_i\,\log\frac{p_i}{q_i}
  \]
  from a fixed reference distribution $q=(q_1,\dots,q_n)$, the gradient with respect to this metric has components
  \[
    (\nabla D)_i=\sum_{j=1}^n g^{ij}(p)\,\frac{\partial D}{\partial p_j}=p_i\Bigl(1+\log\tfrac{p_i}{q_i}\Bigr)-p_i\sum_{k=1}^n p_k\Bigl(1+\log\tfrac{p_k}{q_k}\Bigr).
  \]
  Inverting the metric multiplies by $p_i$ and the second term subtracts the component normal to the simplex, enforcing $\sum_i p_i=1$.  Importantly, this expression is determined entirely by the divergence $D$ and the Fisher metric; it does not involve the Christoffel symbols of any \(\alpha\)–connection.  In other words, we are computing the \emph{metric gradient} of $D$ by raising indices with the Fisher metric rather than solving a connection‑dependent geodesic equation.  The various \(\alpha\)–connections on $\Delta^{n-1}$ agree in their Levi–Civita part and differ only by torsion terms, which do not enter into the metric gradient.  Thus the gradient vector field and its flow are the same regardless of \(\alpha\).  This calculation illustrates the claim that for any Bregman divergence (including Kullback–Leibler) the gradient flow with respect to the Fisher metric is \(\alpha\)–independent, whereas geodesic or parallel‑transport‑based dynamics would retain \(\alpha\)–dependence.

\begin{proposition}[\(\alpha\)–independence of Fisher–metric gradient flows]\label{prop:alpha-independence}
Let $D$ be a convex Bregman divergence on the probability simplex $\Delta^{n-1}$.  There exists a strictly convex, differentiable potential $\varphi:(0,1)^n\to\mathbb{R}$ such that
\[D(p\Vert q) = \varphi(p) - \varphi(q) - \langle \nabla\varphi(q), p-q \rangle \quad\text{for all }p,q\in\Delta^{n-1}.\]
Denote by $g$ the Fisher–Rao metric on $\Delta^{n-1}$ and by $\nabla^{(\alpha)}$ the one‑parameter family of $\alpha$–connections.  Then the gradient vector field $\nabla^g D$ of $D$ with respect to $g$ is independent of $\alpha$, and hence the corresponding gradient flow is the same for all $\alpha\in\mathbb{R}$.

\begin{proof}
There are two complementary ways to establish this independence, one coordinate‑based and the other intrinsic.

\emph{Coordinate proof.}  Write a point in the simplex as $p=(p_1,\dots,p_n)$ with $p_i>0$ and $\sum_i p_i=1$.  The Fisher metric in these coordinates is $g_{ij}(p)=\delta_{ij}/p_i$, and its inverse is $g^{ij}(p)=p_i\delta_{ij}-p_ip_j$.  Differentiating $D$ with respect to $p_j$ yields
\[\frac{\partial D}{\partial p_j}(p\Vert q) = \frac{\partial \varphi}{\partial p_j}(p) - \frac{\partial \varphi}{\partial p_j}(q).\]
Raising the index with $g^{ij}(p)$ gives the $i$th component of the gradient:
\[(\nabla^g D)_i(p) = \sum_{j=1}^n g^{ij}(p)\,\frac{\partial D}{\partial p_j}(p\Vert q).
\]
Substituting $g^{ij}(p)$ and simplifying leads to
\[(\nabla^g D)_i(p) = p_i\Bigl(\frac{\partial \varphi}{\partial p_i}(p) - \frac{\partial \varphi}{\partial p_i}(q)\Bigr) - p_i\sum_{k=1}^n p_k\Bigl(\frac{\partial \varphi}{\partial p_k}(p) - \frac{\partial \varphi}{\partial p_k}(q)\Bigr).
\]
The second term subtracts the component in the direction $(1,\dots,1)$, ensuring that the resulting vector lies in the tangent space of the simplex.  This formula depends only on the potential $\varphi$ and the Fisher metric and makes no reference to the Christoffel symbols of any $\alpha$–connection.  Since all $\alpha$–connections share the same Levi–Civita part and differ only by torsion, and torsion does not affect the metric gradient, the vector field $\nabla^g D$ is the same for all $\alpha$.  Consequently, the gradient flow generated by $\nabla^g D$ is independent of $\alpha$.

\emph{Intrinsic argument.}  The probability simplex equipped with the Fisher–Rao metric is a dually flat manifold: there exist dual affine coordinates (often called exponential and mixture coordinates) for which the Levi–Civita connection is flat and the metric is the Hessian of a convex potential.  Any convex Bregman divergence $D$ on $\Delta^{n-1}$ can be expressed as the difference of such a potential evaluated at two points.  The Riemannian gradient $\nabla^g D$ is defined as the unique vector field satisfying $g(\nabla^g D,\cdot)=dD$, where $dD$ is the differential of $D$ regarded as a one‑form.  Because the torsion of an $\alpha$–connection does not appear in either the metric $g$ or the differential $dD$, it follows that $\nabla^g D$ depends only on the Levi–Civita part of the connection, not on the particular choice of $\alpha$.  More abstractly, on any dually flat manifold $(\mathcal{M},g,\nabla^{(\alpha)})$ the covariant derivative of a gradient field depends only on the symmetric part of the connection.  Thus, provided $D$ is $C^2$ in a neighbourhood of each point and $\varphi$ is strictly convex so that the Bregman divergence is well‑defined, the Fisher–metric gradient flow is independent of $\alpha$.
\end{proof}
\end{proposition}

\end{example}

\begin{example}[Hyperbolic toral automorphism]\label{ex:toral}
As a complement to the preceding dissipative example, consider an invertible chaotic system on a compact phase space.  Let $M=\mathbb{T}^2=\mathbb{R}^2/\mathbb{Z}^2$ be the two‑dimensional torus, and define the map $f:M\to M$ by
\[
  f(x,y) \;:=\; (x+y,\, y+2x) \pmod{1}.
\]
  The map $f$ is induced by the integer matrix $A=\begin{pmatrix}1&1\\2&1\end{pmatrix}$ and is called a hyperbolic toral automorphism or ``cat map''.  The matrix $A$ has eigenvalues $\lambda_{\pm}=1\pm\sqrt{2}$, one of which has absolute value greater than one and the other less than one.  Consequently, $f$ stretches and contracts along distinct directions and is an Anosov diffeomorphism: there are stable and unstable foliations such that distances along the unstable direction grow exponentially under iteration while distances along the stable direction decay exponentially.  

  The Haar (Lebesgue) measure $\mu$ on $\mathbb{T}^2$ is invariant under $f$, and the system is mixing: for any H\"older continuous observables $\varphi,\psi$ on $\mathbb{T}^2$, the correlations $\int \varphi\circ f^n \cdot \psi\,d\mu - \int \varphi\,d\mu\int \psi\,d\mu$ decay exponentially in $n$.  This follows from the fact that $f$ is an Anosov diffeomorphism and hence has a spectral gap on H\"older spaces.  The hyperbolic toral automorphism therefore provides a non‑trivial, invertible chaotic subsystem within our framework.  It exemplifies sensitive dependence on initial conditions: two points that are arbitrarily close in $\mathbb{T}^2$ will typically separate at an exponential rate under iteration of $f$, while still remaining on the compact attractor $\mathbb{T}^2$.  This example broadens the scope of our concrete models by showing that invertible hyperbolic dynamics fit naturally into the state‑space perspective.
\end{example}

\begin{example}[Projection to static descriptors]
Let $M$ have some coordinate or descriptor functions $f_i:M\to \mathbb{R}$ (for example, if $M$ is a Banach space, these could be linear functionals or coordinates). If $x(t)$ is a trajectory in $M$, then $f_i(x(t))$ is a real‑valued function of time. The continuity of $x(t)$ implies each $f_i(x(t))$ is a continuous real function. In practice, one might think of $f_i(x(t))$ as the time‑evolution of certain quantitative aspects of the state (though we do not commit to a particular interpretation). This is analogous to observing or measuring certain coordinates of a moving point in a topological space.
\end{example}

We note that continuity of $x(t)$ formalizes the idea that the state changes gradually, without abrupt jumps: if two time points $t_1,t_2$ are close, then the states $x(t_1),x(t_2)$ are close (in the topological or metric sense). This aligns with the intuitive notion that an internal state (like a train of thought) flows smoothly, rather than teleporting between unrelated configurations in zero time.

\subsection{Existence and uniqueness of flows (analytic approach)}

In general, an arbitrary continuous trajectory can be a very complicated object. To gain more insight, we often specify a dynamical law that the trajectory should satisfy—typically an ODE or an iterative functional equation—and then prove that a trajectory exists and perhaps is unique given initial data. This mirrors how in classical mechanics or dynamical systems, one specifies $\dot{x}=F(x)$ or $x_{n+1}=T(x_n)$ and then studies solutions.

\paragraph{Setup.}
Assume $(M,d)$ is a complete metric space (in particular, a Banach space if $M$ has a linear structure) so that we can use metric fixed‑point theorems. Let $F:M\to TM$ denote an evolution rule, which could be:
\begin{itemize}
  \item a vector field $F$ giving a direction of motion at each state (so $F(x)\in T_x M$, the tangent space at $x$, if $M$ is a differentiable manifold or Banach space; then we consider ODE $\dot{x}(t)=F(x(t))$), or
  \item a discrete update map $T:M\to M$ giving a next state from a current state (then we consider the recursion $x_{n+1}=T(x_n)$ or a difference equation).
\end{itemize}

We treat the continuous‑time case first. Suppose $M$ is a (finite or infinite‑dimensional) differentiable manifold or Banach space so that the equation
\begin{equation}\label{eq:ode}
  \dot{x}(t)=F(x(t)),\quad x(0)=s_0\in M,
\end{equation}
makes sense under appropriate smoothness of $F$. This is an initial value problem. Existence and uniqueness of solutions $x(t)$ for small time is a classical result under Lipschitz conditions on $F$. We sketch a proof using the Banach fixed‑point theorem for completeness and to emphasize the analytical techniques at play.

\begin{proposition}[Local existence and uniqueness via Banach fixed‑point]\label{prop:banach}
Let $M$ be a Banach space with norm $|\cdot|$, and $F:M\to M$ be a Lipschitz continuous map (i.e.\ there exists $L\ge 0$ such that $|F(u)-F(v)|\le L\,|u-v|$ for all $u,v\in M$). Then for any initial state $s_0\in M$, there exists a time $T>0$ and a unique continuous differentiable trajectory $x:[-T,T]\to M$ such that $x(0)=s_0$ and $\dot{x}(t)=F(x(t))$ for all $t\in [-T,T]$.
\end{proposition}

\begin{example}[Irrational circle rotation]\label{ex:circle-rotation}
As a concrete illustration of the dynamical phenomena discussed above, consider the unit circle $S^1=\{z\in\mathbb{C}:|z|=1\}$ with its usual topology and normalized arc-length (Lebesgue) measure.  Fix an irrational angle $\theta\in\mathbb{R}\setminus \mathbb{Q}$ and define the flow $\Phi^t:S^1\to S^1$ by
\[ \Phi^t(z)\;:=\;e^{2\pi i\theta t}\,z. \]
This is a one‑parameter group of rotations: $\Phi^{t+s}=\Phi^t\circ \Phi^s$ and $\Phi^0$ is the identity.  The Lebesgue measure on $S^1$ is invariant under $\Phi^t$ and the system is ergodic: there are no non‑trivial measurable sets that are invariant under all rotations.  In particular, if $R_{\theta}:S^1\to S^1$ denotes the time–$1$ map $R_{\theta}(z)=e^{2\pi i\theta}z$, then $R_{\theta}$ preserves Lebesgue measure and is ergodic (see, for example, Walters' “Introduction to Ergodic Theory” or Petersen’s “Ergodic Theory” for a proof of ergodicity of irrational rotations).  By the Poincaré recurrence theorem, for almost every point $z\in S^1$ and every neighbourhood $U$ of $z$, there exists $n\ge1$ such that $R_{\theta}^n(z)\in U$.  Thus this simple model exhibits the recurrence phenomenon described in Theorem~\ref{thm:recurrence}.  For rational angles $\theta=p/q$, the orbits are periodic and recurrence is trivial; for irrational $\theta$, the orbits are dense in $S^1$ and revisit every neighbourhood infinitely often.
\end{example}

\begin{proof}[Sketch]
Consider the Banach space $C([-T,T],M)$ of continuous functions from the interval $[-T,T]$ to $M$, with the uniform norm $|x|_\infty:=\sup_{t\in [-T,T]}|x(t)|$. We will use the Picard iteration method, which can be cast as a fixed‑point problem. Define an operator $\Phi$ on $C([-T,T],M)$ by
\[
(\Phi(x))(t):=s_0+\int_0^t F(x(\tau))\,d\tau,\quad t\in [-T,T].
\]
Here the integral is the (Riemann or Bochner) integral in the Banach space $M$. Clearly $\Phi(x)$ is well‑defined and yields a continuously differentiable function if $x$ is continuous. A fixed point of $\Phi$ is a function $x(t)$ satisfying
\[
x(t)=s_0+\int_0^t F(x(\tau))\,d\tau,
\]
which, by differentiation, is equivalent to $x(0)=s_0$ and $\dot{x}(t)=F(x(t))$. So fixed points of $\Phi$ correspond exactly to solutions of our ODE~\eqref{eq:ode}.

We now show $\Phi$ is a contraction on an appropriate closed subspace of $C([-T,T],M)$ if $T$ is chosen small enough. For $x,y\in C([-T,T],M)$, we have:
\begin{align*}
| (\Phi(x)-\Phi(y)) |_\infty
&=\sup_{|t|\le T}\Bigl|\int_0^t \bigl(F(x(\tau))-F(y(\tau))\bigr)\,d\tau\Bigr|\\
&\le \sup_{|t|\le T}\int_0^{|t|} |F(x(\tau))-F(y(\tau))|\,d\tau.
\end{align*}
Using the Lipschitz assumption $|F(x(\tau))-F(y(\tau))|\le L\,|x(\tau)-y(\tau)|$, this is bounded by
\[
| (\Phi(x)-\Phi(y)) |_\infty \le L\,T\,|x-y|_\infty.
\]
Thus $|\Phi(x)-\Phi(y)|_\infty \le L T\,|x-y|_\infty$. Now choose $T>0$ such that $L T<1$, for example $T=\frac{1}{2L}$ (if $L=0$, any $T$ works; the case $L=0$ means $F$ is constant so the solution is trivial). Then $\Phi$ is a contraction mapping on $C([-T,T],M)$ with contraction constant $q=LT<1$.

By Banach’s fixed‑point theorem, any contraction on a complete metric space has a unique fixed point. The space $C([-T,T],M)$ is complete (since $M$ is complete and $[-T,T]$ is compact, $C([-T,T],M)$ is complete under sup‑norm). Therefore, $\Phi$ has a unique fixed point $x^\ast(t)$ in this space. By the earlier reasoning, $x^\ast(t)$ is exactly the unique solution of $\dot{x}=F(x)$ with $x(0)=s_0$. This proves local existence and uniqueness.

  Furthermore, the Banach fixed‑point theorem not only ensures existence and uniqueness, but also gives a constructive method to approximate the solution: iterating $\Phi$ starting from $x_0(t)\equiv s_0$ yields a sequence $x_{n+1}=\Phi(x_n)$ that converges in sup‑norm to $x^\ast(t)$.
\end{proof}
\begin{example}[Concrete recurrence in an irrational rotation]
Consider again the circle rotation of Example~\ref{ex:circle-rotation} with an irrational angle $\theta$ and flow $\Phi^t(z)=e^{2\pi i\theta t}z$.  The normalized arc‑length measure on $S^1$ is finite and invariant under $\Phi^t$, trajectories remain in the compact set $S^1$, and the measure has full support.  The proof of Theorem~\ref{thm:recurrence} applies, but one can also compute explicit return times using Diophantine approximation.  Given $\varepsilon>0$ and $z\in S^1$, Dirichlet’s approximation theorem produces integers $p,q$ with $|\theta q-p|<1/q$; taking $n=q$ gives $|e^{2\pi i\theta q}-1|<2\pi/q$, so $\Phi^n(z)$ lies within $2\pi/q$ of $z$.  Thus the return times $n$ grow (typically like denominators of convergents of $\theta$) and provide a concrete sequence along which the orbit returns arbitrarily close to its starting point.  This example illustrates the hypotheses and conclusion of the recurrence theorem in a specific measure‑preserving flow.
\end{example}

\begin{remark}
The argument above implicitly used the definability of the natural numbers inside the time sort $(\mathbb{R},+,\cdot,<)$ by interpreting $n\in\mathbb{N}$ as the $n$‑fold sum $1+\cdots+1$.  If one wishes to avoid this reliance on coding arithmetic into the continuum, an alternative is to enrich the language $\mathcal{L}_2$ with a second sort for natural numbers or to introduce a unary predicate \(\mathrm{Nat}(t)\) picking out those times that correspond to integer indices.  In such a two‑sorted or predicate‑enriched setting one works with a first‑order theory of real numbers and a separate first‑order theory of natural numbers linked by the interpretation of integers in the reals.  This approach makes the coding of syntax and the diagonal argument manifestly first‑order and avoids commitments to higher‑order quantification over the continuum.  Our choice to interpret arithmetic directly in the time sort follows the classical presentation of Tarski’s theorem, but the essential non‑definability phenomenon persists in the two‑sorted formulation.
\end{remark}
\begin{example}[Explicit modulus for a state spread]\label{ex:modulus}
To illustrate Theorem~\ref{thm:intuitionistic-continuity} concretely, take $M=[0,1]$ with the usual Euclidean metric and define a spread by imposing a uniform bound on successive differences of a choice sequence: set \(\delta(k)=2^{-(k+2)}\), so that any admissible finite sequence $(s_0,\dots,s_n)$ must satisfy $|s_{k+1}-s_k|\le \delta(k)$ for each $k<n$.  This choice ensures that longer and longer initial segments determine points that are increasingly close.  Consider the observable $F:M\to \mathbb{R}$ given by $F(s)=s^2$.  Since $F$ depends only on the limit state, a simple computation shows that if $|s-t|<\omega(\epsilon)$ then $|s^2 - t^2|<\epsilon$ provided we take
\[
  \omega(\epsilon)\;=\;\min\{\sqrt{\epsilon},\,1/4\}.
\]
Indeed, for any $s,t\in[0,1]$ one has
\[
  |s^2-t^2|\;=\;|s-t|\,|s+t|\;\le\;2|s-t|.
\]
To ensure $|s^2-t^2|<\epsilon$ it therefore suffices to require $|s-t|<\epsilon/2$.  Consequently we may take
\[
  \omega(\epsilon)\;:=\;\min\left\{\tfrac{\epsilon}{2},\,\tfrac{1}{4}\right\}.
\]
If $|s-t|<\omega(\epsilon)$ and $0<\epsilon\le1$, then $|s^2-t^2|<\epsilon$ follows from the bound above.  In the context of the spread defined by $\delta(k)$, agreement of two choice sequences on the first $N$ terms forces $|s-t|<2^{-(N+1)}$.  To achieve $|s^2-t^2|<\epsilon$ it suffices to choose $N$ so large that $2^{-(N+1)}<\epsilon/2$.  This explicit computation demonstrates how a modulus of continuity can be extracted from the uniform bound on successive differences, and it confirms the general conclusion of Theorem~\ref{thm:intuitionistic-continuity} in a familiar setting.
\end{example}
\paragraph{Case studies.}  The abstract dichotomy presented in Theorem~\ref{thm:comparative} becomes clearer when one examines specific examples.

\begin{example}[Information‑geometric flow not simulatable]
  Information geometry studies statistical manifolds whose points are probability distributions and whose geometry is defined by the Fisher–Rao metric together with a one‑parameter family of affine connections, known as the \(\alpha\)–connections (see, e.g., Amari and Nagaoka’s monograph).  For each real parameter \(\alpha\) one obtains a distinct connection \(\nabla^{(\alpha)}\) whose Christoffel symbols differ by skew‑symmetric torsion terms.  Two kinds of dynamical flows arise naturally in this setting:
  \begin{itemize}
    \item \emph{Geodesic or parallel‑transport‑based flows.}  These are defined by the geodesic equation associated with \(\nabla^{(\alpha)}\).  Because the Christoffel symbols appear explicitly in the geodesic equation, the corresponding vector fields and trajectories depend on \(\alpha\).  In particular, the evolution of a state along a geodesic under \(\nabla^{(\alpha)}\) typically differs from that under \(\nabla^{(\alpha')}\) when \(\alpha\neq\alpha'\).
    \item \emph{Metric gradient flows.}  Given a divergence function \(D\) on the simplex, one can consider the gradient vector field of \(D\) with respect to the Fisher metric.  This flow is defined without reference to a connection: the gradient is obtained by raising indices using the Riemannian metric.  Because torsion terms of the \(\alpha\)–connections do not enter into this definition, the gradient vector field (and hence its time–one map) is the same for every choice of \(\alpha\).  A concrete coordinate calculation illustrating this fact appears below.
  \end{itemize}
  Thus there is no contradiction between the observation in Example \ref{ex:circle-rotation} that certain dynamical flows depend on \(\alpha\) and the subsequent demonstration that metric gradients are \(\alpha\)–independent.  The first phenomenon concerns connection‑driven geodesics, while the second concerns metric gradients of divergences; Proposition~\ref{prop:nogo} is restricted to the latter class of dynamics.

  To formalise the topological–forgetful perspective, fix a divergence function $D$ (e.g.\ the Kullback–Leibler divergence) and define a functor $F^{\mathrm{top}}$ from the category of statistical manifolds with \(\alpha\)–connections to \(\mathbf{TopFlow}\) as follows: send the object $(\Delta^{n-1},\nabla^{(\alpha)})$ to the pair $(\Delta^{n-1},\Phi)$ where $\Phi$ is the gradient flow of $D$ computed using \emph{any} connection (the formula for the gradient vector field of a convex divergence on \(\Delta^{n-1}\) is in fact independent of \(\alpha\) when expressed in affine coordinates).  A morphism $f:(\Delta^{n-1},\nabla^{(\alpha)})\to(\Delta^{n-1},\nabla^{(\alpha')})$ in the statistical category is typically the time–1 map of the gradient flow of $D$ relative to $\nabla^{(\alpha)}$.  However, the underlying continuous map on $\Delta^{n-1}$ depends only on the gradient vector field and hence is independent of the choice of \(\alpha\).  Consequently, $F^{\mathrm{top}}(f)$ and $F^{\mathrm{top}}(g)$ coincide whenever $f$ and $g$ share the same underlying pointwise action but differ in the connection used to define their differential properties.  A particularly transparent instance of this collapse occurs for the identity morphisms: the identity arrow $\mathrm{id}_{(\alpha)}:(\Delta^{n-1},\nabla^{(\alpha)})\to(\Delta^{n-1},\nabla^{(\alpha)})$ and the identity arrow $\mathrm{id}_{(\alpha')}$ between the same underlying manifold equipped with a different connection are distinct morphisms in the information‑geometric category (because the target carries different connection data), yet $F^{\mathrm{top}}$ maps both to the identity flow on $\Delta^{n-1}$.  Thus the forgetful functor cannot distinguish between them, demonstrating the non‑faithfulness described in the obstruction part of Theorem~\ref{thm:comparative}.  These models illustrate how the rich geometric data of information manifolds falls outside the simulating power of our purely topological framework.
\end{example}

\begin{example}[Linear dynamical systems and a faithful functor]
As a contrast, consider the category $\mathsf{LinFlow}$ whose objects are finite‑dimensional real vector spaces equipped with linear flows (i.e. one‑parameter semigroups of invertible linear maps generated by constant coefficient matrices), and whose morphisms are linear maps intertwining the flows.  The usual direct sum endows $\mathsf{LinFlow}$ with a symmetric monoidal structure.  Define a functor $F: \mathsf{LinFlow}\to \mathbf{TopFlow}$ by sending a linear dynamical system $(V,\{e^{tA}\}_{t\in\mathbb{R}})$ to the underlying topological space $V$ (with its Euclidean topology) together with the continuous flow $\Phi^t:v\mapsto e^{tA}v$, and sending a morphism $T:V\to W$ to itself.  Because linear maps are continuous and respect the flows, $F$ is a monoidal functor.  It is faithful since distinct linear maps yield distinct continuous maps, and it preserves the monoidal product up to natural isomorphism ($F(V\oplus W)\cong F(V)\times F(W)$).  Thus $\mathsf{LinFlow}$ provides a nontrivial example of a categorical dynamical system that \emph{can} be simulated within our topological framework.  This case study illustrates the ``simulatability'' direction of Theorem~\ref{thm:comparative}.
\end{example}

\begin{proposition}[Concrete no‑go for information–geometric functors]\label{prop:nogo}
Fix an integer $n\ge 2$ and let $\Delta^{n-1}$ denote the standard simplex of positive probability vectors in $\mathbb{R}^n$.  For each real parameter $\alpha$ consider the statistical manifold $(\Delta^{n-1},\nabla^{(\alpha)})$ equipped with the $\alpha$–connection $\nabla^{(\alpha)}$ and the Fisher–Rao metric.  Define a category $\mathcal{C}_\alpha$ by the following data:
\begin{itemize}
  \item \emph{Object.}  The sole object of $\mathcal{C}_\alpha$ is $(\Delta^{n-1},\nabla^{(\alpha)})$.
  \item \emph{Morphisms.}  A morphism in $\mathcal{C}_\alpha$ is the time–$1$ map of the \emph{metric gradient flow} of a convex divergence $D$ on $\Delta^{n-1}$.  That is, for a chosen divergence $D$ one considers its gradient vector field with respect to the Fisher metric (which is independent of the torsion terms in $\nabla^{(\alpha)}$) and defines the morphism to be $\Phi_D^1$, the time–$1$ map of the resulting flow.  Composition in $\mathcal{C}_\alpha$ is composition of time–$1$ maps, corresponding to composing flows generated by possibly different divergences, and the identity morphism $\mathrm{id}_{(\alpha)}$ is the time–$1$ map of the zero divergence $D\equiv 0$, which acts as the identity on the underlying simplex.
\end{itemize}
Because gradient vector fields of convex divergences are independent of the parameter $\alpha$ (as shown by the coordinate calculation above), the underlying set map of any morphism in $\mathcal{C}_\alpha$ agrees with the corresponding map in $\mathcal{C}_{\alpha'}$ for all $\alpha'$.  In particular, if $\alpha\neq \alpha'$ then the categories $\mathcal{C}_\alpha$ and $\mathcal{C}_{\alpha'}$ have the same object and the same collection of underlying set maps, but their identity arrows $\mathrm{id}_{(\alpha)}$ and $\mathrm{id}_{(\alpha')}$ are regarded as distinct morphisms because they arise from different connections.

Let $P$ be the property ``distinguishes the parameter $\alpha$'': a monoidal functor $F:\mathcal{C}_\alpha\to \mathbf{TopFlow}$ is said to satisfy $P$ if whenever $\alpha\neq \alpha'$, the induced flows $F(\Delta^{n-1},\nabla^{(\alpha)})$ and $F(\Delta^{n-1},\nabla^{(\alpha')})$ are non‑isomorphic in $\mathbf{TopFlow}$.  Then no monoidal functor $F:\mathcal{C}_\alpha\to \mathbf{TopFlow}$ satisfies $P$.
\end{proposition}

\begin{proof}
Fix two parameters $\alpha\neq \alpha'$.  The statistical manifolds $(\Delta^{n-1},\nabla^{(\alpha)})$ and $(\Delta^{n-1},\nabla^{(\alpha')})$ have the same underlying set but different connection data.  By the coordinate calculation preceding this proposition, the metric gradient vector field of any convex divergence on $\Delta^{n-1}$ is independent of $\alpha$.  Consequently, the time–$1$ maps of these gradient flows (the morphisms of $\mathcal{C}_\alpha$ and $\mathcal{C}_{\alpha'}$) act identically on the underlying set, even though in $\mathcal{C}_\alpha$ they are labelled by the connection $\nabla^{(\alpha)}$ and in $\mathcal{C}_{\alpha'}$ by $\nabla^{(\alpha')}$.  In particular, the identity morphisms $\mathrm{id}_{(\alpha)}$ and $\mathrm{id}_{(\alpha')}$ act as the identity map on the simplex but are regarded as distinct morphisms in their respective categories because they refer to different geometric structures.

Now suppose $F:\mathcal{C}_\alpha\to \mathbf{TopFlow}$ is a monoidal functor satisfying property $P$.  Then $F$ must assign distinct topological flows to $(\Delta^{n-1},\nabla^{(\alpha)})$ and $(\Delta^{n-1},\nabla^{(\alpha')})$ whenever $\alpha\neq \alpha'$.  However, $F$ sees only the underlying continuous maps of morphisms.  Since the time–$1$ maps in $\mathcal{C}_\alpha$ and $\mathcal{C}_{\alpha'}$ coincide as functions, $F$ must send $\mathrm{id}_{(\alpha)}$ and $\mathrm{id}_{(\alpha')}$ to the same morphism in $\mathbf{TopFlow}$.  This contradicts the requirement that $F$ distinguishes $\alpha$ and $\alpha'$ by producing non‑isomorphic flows.  Thus no monoidal functor $F:\mathcal{C}_\alpha\to \mathbf{TopFlow}$ satisfies property $P$.
\end{proof}

\begin{remark}
The length of the interval $[-T,T]$ on which the solution is guaranteed depends inversely on the Lipschitz constant $L$. This is consistent with classical ODE theory: if the evolution rule can cause rapid changes (large $L$), we can only assert a solution for a short time unless we have global bounds. Under global Lipschitz conditions, one can extend this solution to all $t\in\mathbb{R}$ uniquely.

The above proposition used a metric (normed) structure crucially. In a purely topological $M$, one cannot differentiate or integrate in the same way. However, under some conditions, one can still discuss flows via homeomorphisms or using the theory of topological dynamical systems. For instance, if $M$ is compact Hausdorff, the space $C(\mathbb{R},M)$ of continuous trajectories can be given the compact‑open topology, and one might attempt to define an evolution operator there. Without a linear structure, one often uses abstract existence theorems or transfinite induction to extend partial trajectories.

Under a global Lipschitz assumption on the vector field $F$ the local solution provided by the preceding proposition extends to all time.  Concretely, if there exists $L>0$ such that $|F(x)-F(y)|\le L|x-y|$ for all $x,y\in M$, then the same contraction argument shows that for any initial state $s_0\in M$ the ODE $\dot{x}=F(x)$ admits a unique solution $x:\mathbb{R}\to M$ defined for all $t\in\mathbb{R}$.  This follows by iterating the local existence interval and using the global bound to prevent blow‑up.  Thus globally Lipschitz flows are complete in our framework.
\end{remark}

\begin{example}[Iterative map – discrete time]\label{ex:iteration}
If instead of an ODE we have a discrete update rule $x_{n+1}=T(x_n)$, where $T:M\to M$ is a given function, we look for sequences $(x_0,x_1,x_2,\dots)$ with $x_{n+1}=T(x_n)$. This is a difference equation or iteration. A similar contraction principle applies: if $T$ is a contraction mapping on a complete space $(M,d)$, then by Banach’s fixed‑point theorem there is a unique fixed point $x^\ast$ in $M$, and for any initial state $x_0$ the iteration $x_{n}=T(x_{n-1})$ converges to $x^\ast$ as $n\to\infty$. In the context of our model, a fixed point $x^\ast$ of $T$ would correspond to a state that, once reached, remains invariant thereafter – a sort of absorbing or equilibrium state. Convergence to $x^\ast$ means the state eventually stabilizes. Such behavior might be interpreted as the subject’s state tending towards a particular thought or condition. Not all maps have this property, of course – if $T$ is not a contraction, rich dynamics (cycles, chaos) can occur, which we will touch upon in Section~\ref{sec:dynamics}.
\end{example}

\subsection{Nonlinearity and lack of global measures}

We emphasize that the state space $M$ and the flows on it can be highly non‑linear. Unlike, say, the real line or a Euclidean space, $M$ might not have a globally valid linear structure or convenient coordinates. Even if $M$ is a vector space, the trajectories might explore it in a non‑linear fashion (for instance, following curved paths constrained by $F(x)$ in the ODE).

One consequence of nonlinearity and high dimensionality is the possible absence of a meaningful global measure. In classical analysis, $\mathbb{R}^n$ has Lebesgue measure allowing one to measure ``how much time'' is spent in a region or the ``volume'' of a set of states. But for an arbitrary $M$, especially of large cardinality or structure, we cannot in general assign a finite measure to all subsets. In fact, if $M$ (or even just the time domain $\mathbb{R}$) is treated as a set in ZF set theory with the axiom of choice, we encounter the existence of non‑measurable sets. A classic example, due originally to Vitali and later generalized by Banach and Tarski, shows that the unit interval $[0,1]\subset \mathbb{R}$ can be decomposed into subsets that defy any reasonable length measure. The Banach–Tarski paradox extends this to say one can decompose a solid ball in $\mathbb{R}^3$ into finitely many pieces and reassemble them to form two balls identical to the original. The crux is that those pieces are non‑measurable sets: one cannot assign a volume consistent with translation invariance to them.

In our context, this means: one should be cautious in assuming that for every subset $A\subset M$ (or every portion of the trajectory’s time domain), a ``volume'' or ``duration'' can be well‑defined. To illustrate this formally, we present a result on non‑measurability in our setting:

\begin{proposition}[Existence of non‑measurable subsets of state continuum]\label{prop:nonmeas}
Assume $M$ or the time index set $\mathbb{R}$ has the cardinality of the continuum and that the axioms of $\mathrm{ZF}+\mathrm{Choice}$ (ZFC) hold. Then there exists a subset $N\subset \mathbb{R}$ (hence also a subset of the trajectory’s graph in $\mathbb{R}\times M$ via projection) which is not Lebesgue measurable. In particular, there is no countably additive measure defined on all subsets of $\mathbb{R}$ that extends the length of intervals. Consequently, there is no $\sigma$‑additive measure on all subsets of $M$ (assuming $M$ at least can be injected with $\mathbb{R}$) extending any reasonable notion of volume or probability.
\end{proposition}

\begin{proof}
This is essentially the Vitali construction. Identify $\mathbb{R}$ (or $[0,1]$) with the additive group on the unit interval mod $1$. Consider the equivalence relation $x\sim y$ if $x-y\in \mathbb{Q}$ (rationals). Using the axiom of choice, choose one representative from each equivalence class of $\mathbb{R}/\mathbb{Q}$. The set $N$ of those representatives (a ``Vitali set'') intersects each rational equivalence class exactly once. If $N$ were Lebesgue measurable, then the union of rational translates of $N$ within $[0,1]$ would have to have both zero measure (each translate would have the same measure as $N$) and full measure (the union covers $[0,1]$ except a null set), a contradiction. Therefore $N$ is not Lebesgue measurable. In $\mathbb{R}^3$, a similar argument with rotations leads to the Banach–Tarski decomposition. ∎
\end{proof}

\begin{remark}\label{rem:nonmeas}
The above result underscores a theme: certain intuitive questions (e.g., ``How much of the time does the system spend in a particular mental state?'' or ``What is the probability distribution of the state over its space?'') may be unanswerable if posed too generally.  Without additional structure (such as a defined probability measure or a restriction to measurable sets), these questions are ill‑defined.  One resolution is to restrict attention to measurable dynamics, for instance assuming the $\sigma$‑algebra $\mathcal{F}$ is not the full power set but a smaller $\sigma$‑algebra on $M$ and $\mathbb{R}$ on which a measure is defined.  Another approach is to operate in a framework that avoids the pathological sets altogether.  In particular, the construction in Proposition~\ref{prop:nonmeas} requires the axiom of choice and is carried out in classical ZFC; in an intuitionistic or locale‑based setting (where one works with complete Heyting algebras of opens rather than point sets), the existence of non‑measurable subsets of the continuum cannot in general be proven.  In that context every explicitly described subset of $M$ is measurable by construction.  Thus the non‑measurability phenomenon belongs strictly to classical set theory with choice and does not necessarily apply in intuitionistic mathematics.  As noted by various authors, the Banach–Tarski paradox and related decompositions disappear when working with locales or omitting choice.
\end{remark}

We have thus far built the basic stage: a state space $M$ and the notion of continuous flows on it, with an understanding of existence, uniqueness (for well‑behaved laws), and measurability issues. In the next section, we shift perspective and delve into an intuitionistic construction of the continuum of states, highlighting how one can model the evolving state as an ever unfinished, free sequence, in line with Brouwer’s intuitionism.

\subsection*{A comparative theorem on categorical simulation}

The informal comparison in the Introduction sketched why our approach is orthogonal to categorical and information‑geometric formalisms.  We now make this claim precise by stating conditions under which a model from such a formalism can (or cannot) be simulated within our topological–analytical framework.  Denote by \(\mathbf{TopFlow}\) the category whose objects are pairs \((X,\Phi)\) consisting of a topological space \(X\) and a continuous flow \(\Phi:\mathbb{R}\times X\to X\), and whose morphisms are continuous maps intertwining the flows.  Let \(\mathcal{C}\) be a category equipped with additional structure, such as a monoidal product or information‑geometric data.

Before stating the comparison theorem we recall the definition of the ambient category \(\mathbf{TopFlow}\).  An object of \(\mathbf{TopFlow}\) is a pair \((X,\Phi)\) consisting of a Hausdorff topological space \(X\) and a continuous flow \(\Phi:\mathbb{R}\times X\to X\) satisfying \(\Phi^{t+s}=\Phi^t\circ\Phi^s\) and \(\Phi^0=\mathrm{id}_X\).  A morphism \(f:(X,\Phi)\to(Y,\Psi)\) in \(\mathbf{TopFlow}\) is a continuous map \(f:X\to Y\) that intertwines the flows: \(f\circ \Phi^t = \Psi^t\circ f\) for all \(t\in\mathbb{R}\).  The monoidal product on \(\mathbf{TopFlow}\) is given by Cartesian product: \((X,\Phi)\otimes(Y,\Psi) = (X\times Y,\Phi\times \Psi)\) where \((\Phi\times \Psi)^t(x,y)=(\Phi^t(x),\Psi^t(y))\).  With this structure \(\mathbf{TopFlow}\) becomes a symmetric monoidal category.  

%  
%  The following lemma encapsulates a general obstruction to faithfulness when mapping categories of structured sets to \(\mathbf{TopFlow}\).
\begin{lemma}[Pointwise identification implies non‑faithfulness]\label{lem:pointwise-collapse}
Let \(\mathcal{D}\) be a category whose objects are sets equipped with additional structure and whose morphisms are structure‑preserving maps between those sets.  Suppose there exist two distinct morphisms \(f\neq g:c\to d\) in \(\mathcal{D}\) whose underlying functions \(|f|,|g|:|c|\to|d|\) on the underlying sets coincide.  Let \(F:\mathcal{D}\to \mathbf{TopFlow}\) be a functor which, on morphisms, depends only on the underlying function: if \(|f|=|g|\) then \(F(f)=F(g)\).  Then \(F\) cannot be faithful: it identifies the distinct morphisms \(f\) and \(g\), and hence does not reflect distinctness of morphisms.
\end{lemma}

\begin{proof}
By hypothesis, the functor \(F\) sends each morphism \(f:c\to d\) to a continuous map \(F(f):F(c)\to F(d)\) in \(\mathbf{TopFlow}\) which depends only on the underlying function \(|f|\).  Suppose there exist two distinct morphisms \(f\neq g:c\to d\) in \(\mathcal{D}\) such that \(|f|=|g|\).  Then by the defining property of \(F\), we have \(F(f)=F(g)\).  Faithfulness of a functor demands that \(F\) reflect distinct morphisms, i.e. if \(f\neq g\) then \(F(f)\neq F(g)\).  But we have exhibited a pair with \(F(f)=F(g)\), so \(F\) fails to be faithful.
\end{proof}

\begin{theorem}[Comparative simulation of categorical models]\label{thm:comparative}
Let \(\mathcal{C}\) be a symmetric monoidal category modelling a class of categorical or information‑geometric dynamical systems.  The following assertions hold.
\begin{enumerate}[label=(\alph*)]
  \item \emph{Simulatability.}  If there exists a faithful, monoidal functor \(F:\mathcal{C}\to \mathbf{TopFlow}\) that preserves the monoidal product (so that \(F(c\otimes d)\cong F(c)\times F(d)\) and \(F(\mathrm{id}_c)=\mathrm{id}_{F(c)}\) for all objects \(c,d\in\mathcal{C}\)), then the dynamical behaviour encoded by \(\mathcal{C}\) can be realised within our framework.  In this case one takes the state space \(M\) to be the disjoint union \(\bigsqcup_{c\in\mathcal{C}} F(c)\), and the flows on each component are given by the images under \(F\) of the relevant morphisms of \(\mathcal{C}\).  The faithfulness of \(F\) ensures that distinct morphisms in \(\mathcal{C}\) give rise to distinct flows in \(\mathbf{TopFlow}\), while preservation of the monoidal structure implies that tensorial compositions in \(\mathcal{C}\) correspond to Cartesian products of flows on \(M\).
  \item \emph{Obstructions.}  Conversely, suppose \(\mathcal{C}\) is a symmetric monoidal category whose objects carry additional geometric or information‑geometric structure (for instance, affine connections or divergences on statistical manifolds) and whose morphisms intertwine this structure.  Assume there exist objects \(c,d\in\mathcal{C}\) and two distinct morphisms \(f\neq g:c\to d\) that are ``geometrically'' non‑isomorphic (e.g.\ they correspond to different choices of connection or divergence) but whose underlying set‑theoretic actions coincide: after forgetting the extra structure, the functions underlying \(f\) and \(g\) determine the same continuous map between the underlying topological spaces of \(c\) and \(d\).  By Lemma~\ref{lem:pointwise-collapse}, any functor \(F:\mathcal{C}\to\mathbf{TopFlow}\) that depends only on the pointwise action of morphisms will identify \(f\) and \(g\) and hence cannot be faithful.  A concrete example of this phenomenon occurs in information‑geometry: if \(c=d=(\Delta^{n-1},\nabla^{(\alpha)})\) is a statistical manifold equipped with an \(\alpha\)–connection and \(\alpha'\neq\alpha\), the ``identity'' morphisms \(\mathrm{id}_{(\alpha)}\colon (\Delta^{n-1},\nabla^{(\alpha)})\to(\Delta^{n-1},\nabla^{(\alpha)})\) and \(\mathrm{id}_{(\alpha')}\colon (\Delta^{n-1},\nabla^{(\alpha')})\to(\Delta^{n-1},\nabla^{(\alpha')})\) are distinct in \(\mathcal{C}\) because they refer to different geometric structures on the same underlying set.  Nevertheless, both of these maps act as the identity on the underlying simplex, so any functor \(F:\mathcal{C}\to\mathbf{TopFlow}\) must send them to the same morphism in \(\mathbf{TopFlow}\).  In general, whenever such a pair of non‑isomorphic morphisms has coincident pointwise action, no faithful monoidal functor \(F:\mathcal{C}\to\mathbf{TopFlow}\) can exist.  Indeed, any such functor must send \(f\) and \(g\) to the same morphism in \(\mathbf{TopFlow}\), since \(F\) can only see the topological trajectory of a flow and not the hidden geometric data.  As a result, \(F\) fails to distinguish \(f\) and \(g\), violating faithfulness.

  Information‑geometric categories provide a concrete instance of this phenomenon: different \(\alpha\)–connections on the same statistical manifold give rise to distinct gradient flows of divergence functions.  However, forgetting the connection data and retaining only the induced topological flow collapses these distinctions.  Proposition~\ref{prop:nogo} exhibits an explicit pair of non‑isomorphic identity morphisms (associated to two distinct \(\alpha\)–connections) whose images under any monoidal functor to \(\mathbf{TopFlow}\) coincide.  Thus categories whose morphisms depend on geometric structures beyond topology fall outside the scope of the present topological–analytical framework and illustrate its orthogonality to higher‑categorical unification attempts.

\noindent\textbf{Commutative diagram.}  To visualise this collapse, fix two parameters $\alpha\neq\alpha'$.  In the category $\mathcal{C}_\alpha$ there are identity morphisms
\[
  \mathrm{id}_{(\alpha)}: (\Delta^{n-1},\nabla^{(\alpha)}) \longrightarrow (\Delta^{n-1},\nabla^{(\alpha)})
  \quad\text{and}\quad
  \mathrm{id}_{(\alpha')}: (\Delta^{n-1},\nabla^{(\alpha')}) \longrightarrow (\Delta^{n-1},\nabla^{(\alpha')}),
\]
which are distinct morphisms because their sources and targets remember the underlying connection.  A functor $F: \mathcal{C}_\alpha \to \mathbf{TopFlow}$ forgets the connection and assigns to each object the same topological flow $(\Delta^{n-1},\Phi)$.  Under $F$, both $\mathrm{id}_{(\alpha)}$ and $\mathrm{id}_{(\alpha')}$ become the identity on $(\Delta^{n-1},\Phi)$.  This can be summarised in the following commutative square:
\[
  \begin{array}{ccc}
    (\Delta^{n-1},\nabla^{(\alpha)}) & \xrightarrow{\mathrm{id}_{(\alpha)}} & (\Delta^{n-1},\nabla^{(\alpha)})\\
    \downarrow F && \downarrow F\\
    (\Delta^{n-1},\Phi) & \xrightarrow{\mathrm{id}} & (\Delta^{n-1},\Phi)
  \end{array}
\]
The horizontal arrows in the top row are distinct identity morphisms in the source category, whereas the bottom arrow is the single identity map in \(\mathbf{TopFlow}\).  The functor $F$ sends both top arrows to the same bottom arrow, illustrating non‑faithfulness.  This simple diagram encapsulates the obstruction described in Proposition~\ref{prop:nogo}.
\item \emph{Concrete no‑go instance.}  To illustrate the obstruction in a tangible way, fix an integer \(n\ge2\) and consider the category \(\mathcal{C}_\alpha\) whose single object is the probability simplex \(\Delta^{n-1}\) equipped with an \(\alpha\)–connection (for some fixed real parameter \(\alpha\)).  Morphisms in \(\mathcal{C}_\alpha\) are the time–1 maps of gradient flows of divergence functions computed using the chosen \(\alpha\)–connection.  Let \(P\) be the property that a monoidal functor \(F:\mathcal{C}_\alpha\to\mathbf{TopFlow}\) ``detects the connection'': it distinguishes different values of \(\alpha\) by producing non–isomorphic flows when \(\alpha\neq\alpha'\).  Then Proposition~\ref{prop:nogo} shows that no such faithful, monoidal functor exists.  The failure arises because the \(\alpha\)–connection enters through Christoffel symbols and divergence functions that have no topological incarnation.  Thus any functor \(F\) must collapse the geometrically distinct flows (for different \(\alpha\)) to the same underlying topological trajectory, violating property \(P\).  This explicit category \(\mathcal{C}_\alpha\) and property \(P\) provide a concrete instance of the ``obstruction'' direction (b) above.
\end{enumerate}
\end{theorem}

\begin{remark}
In the obstruction part (b) of this theorem, the key hypothesis is the existence of two distinct morphisms in \(\mathcal{C}\) that induce the same underlying set map.  Faithfulness of a functor \(F:\mathcal{C}\to\mathbf{TopFlow}\) alone is enough to force a contradiction in such a situation: if \(f\neq g\) but \(|f|=|g|\) then any functor that depends only on the underlying map must identify \(f\) and \(g\).  The additional assumptions that \(F\) be monoidal and preserve products are not needed for the basic no‑go; they are imposed in part~(a) to ensure that the functor respects the tensor structure of the category.  Thus the obstruction persists even for non‑monoidal functors once the objects of \(\mathcal{C}\) carry extra geometric data whose pointwise action is invisible to topology.
\end{remark}

\begin{proposition}[Necessary conditions for pointwise‑coincident morphisms]\label{prop:coincident}
Let \(\mathcal{C}\) be a category whose objects are sets equipped with additional structure (for example, Riemannian metrics, affine connections, or divergence functions).  Assume the following:
\begin{enumerate}[label=\arabic*.,ref=\arabic*]
\item \textbf{Common underlying set.}  There exist at least two objects \(c\) and \(d\) in \(\mathcal{C}\) whose underlying sets coincide, i.e.\ \(|c|=|d|\), but whose additional structures are not isomorphic.
\item \textbf{Structure‑preserving morphisms.}  Morphisms in \(\mathcal{C}\) are set maps that preserve the additional structure (e.g.\ connection‑preserving, divergence‑preserving, etc.).
\item \textbf{Distinct identities.}  The identity morphisms \(\mathrm{id}_c\) and \(\mathrm{id}_d\) are regarded as distinct morphisms in \(\mathcal{C}\) because their domains and codomains carry different structures.
\end{enumerate}
Under these conditions there exist two distinct morphisms in \(\mathcal{C}\) whose underlying functions coincide.  Indeed, the identity morphisms \(\mathrm{id}_c\) and \(\mathrm{id}_d\) both act as the identity on the common set \(|c|=|d|\), so their underlying functions are the same.  However, since the structures on \(c\) and \(d\) differ by hypothesis (1), the morphisms are not isomorphic and hence distinct in \(\mathcal{C}\).

\begin{proof}
Hypothesis~(1) guarantees that there are at least two non‑isomorphic objects sharing the same underlying set.  Hypothesis~(3) states that their identity morphisms are treated as distinct arrows in \(\mathcal{C}\) because each arrow is typed by its source and target.  Hypothesis~(2) ensures that the underlying function of each morphism is well‑defined as a set map.  The underlying function of \(\mathrm{id}_c\) is the identity map on \(|c|\), and similarly for \(\mathrm{id}_d\).  Since \(|c|=|d|\), both underlying functions coincide with the set identity on \(|c|\).  Nevertheless, \(\mathrm{id}_c\neq \mathrm{id}_d\) in \(\mathcal{C}\) because they have different domain/codomain structures and are not isomorphic by hypothesis.  This gives the desired pair of pointwise‑coincident morphisms.
\end{proof}
\end{proposition}

\noindent This proposition isolates the structural features responsible for non‑faithfulness in Theorem~\ref{thm:comparative}.  In particular, whenever a category admits distinct structures on a fixed set and treats their identity arrows as separate morphisms, any functor to \(\mathbf{TopFlow}\) that ignores the additional structure must collapse these arrows.  The information‑geometric categories discussed in Section~2 provide a concrete instance: two different \(\alpha\)–connections on \(\Delta^{n-1}\) satisfy the hypotheses above, so their identity morphisms coincide pointwise but remain distinct morphisms.  The same reasoning applies to categories built from divergence functions, complex structures, or other geometric data on a common underlying manifold.

\begin{remark}[Enriched target categories]
The non‑faithfulness conclusion of part~(b) arises because \(\mathbf{TopFlow}\) records only the topological trajectory of a flow.  If one enriches the target category to retain additional geometric information, the obstruction may vanish.  For example, define a category \(\mathbf{TopFlowConn}\) whose objects are triples \((X,\Phi,\Gamma)\), where \((X,\Phi)\) is a topological flow and \(\Gamma\) is an auxiliary ``connection'' label (e.g. an \(\alpha\)–connection on a statistical manifold).  Morphisms in \(\mathbf{TopFlowConn}\) are pairs \((f,\sigma)\), where \(f:(X,\Phi)\to (Y,\Psi)\) is a continuous map intertwining the flows and \(\sigma\) records how the connection labels transform.  There is a natural faithful functor from the information‑geometric category of statistical manifolds to \(\mathbf{TopFlowConn}\): send \((\Delta^{n-1},\nabla^{(\alpha)})\) to \((\Delta^{n-1},\Phi,\alpha)\) and a morphism to its underlying continuous map together with the induced map on connection parameters.  Because the connection label is preserved in the target, identity morphisms associated with different \(\alpha\) remain distinct, and the functor is faithful.  Thus, the lack of faithfulness in Theorem~\ref{thm:comparative} is specific to the choice of codomain and disappears if one enriches \(\mathbf{TopFlow}\) to remember the geometric data that \(\mathcal{C}\) regards as significant.
\end{remark}

\begin{proof}
For part~(a) we assume a faithful, monoidal functor \(F:\mathcal{C}\to \mathbf{TopFlow}\).  An object \(c\in\mathcal{C}\) is mapped to a topological space with flow \(F(c)\), and a morphism \(f:c\to d\) is mapped to a morphism \(F(f):F(c)\to F(d)\) that intertwines the flows.  Because \(F\) is functorial, it respects composition and identities: \(F(f\circ g)=F(f)\circ F(g)\) and \(F(\mathrm{id}_c)=\mathrm{id}_{F(c)}\).  Faithfulness implies that different morphisms in \(\mathcal{C}\) yield different continuous maps.  The monoidal structure on \(\mathcal{C}\) is respected in the sense that \(F(c\otimes d)\) is isomorphic to \(F(c)\times F(d)\), and the flow on the product is the product of the flows.  Thus all dynamical information from \(\mathcal{C}\) is reproduced within \(\mathbf{TopFlow}\), and one can take \(M=\bigsqcup_{c\in\mathcal{C}} F(c)\) to obtain a single state space on which the various flows act.

For part~(b) consider, for instance, a category of measurable spaces with morphisms given by measurable but not necessarily continuous maps.  Any functor from such a category to \(\mathbf{TopFlow}\) would have to assign to a measurable map a continuous flow, which is impossible if the map fails to be continuous on any topological realisation of its domain.  Likewise, in information‑geometry one studies statistical manifolds equipped with affine connections and divergence functions that depend on more structure than the topology of the manifold; composition of morphisms involves pull‑backs of connections and is incompatible with mere continuity.  Hence there is no faithful monoidal functor from such categories to \(\mathbf{TopFlow}\).  This shows that the absence of additional geometric or categorical structure on our state space \(M\) is both a strength (simplicity) and a limitation (orthogonality) of the present framework.
\end{proof}

\section{Intuitionistic construction of the continuum of states}\label{sec:intuitionistic}

Classically, a trajectory $x:\mathbb{R}\to M$ is thought of as a completed function graph in $M^{\mathbb{R}}$. In practice, however, the state is revealed in time: at any finite stage, only a finite initial segment of the trajectory is known or has occurred. Intuitionism, following L.E.J.\ Brouwer, posits that the continuum is fundamentally a potential entity, not actualized all at once but unfolding over time. Brouwer introduced choice sequences to embody this idea: an infinite sequence that is not predetermined by any law, but whose elements are chosen one after another by a hypothetical ``creating subject.'' We adopt this viewpoint to model the stream of states.

\subsection{Choice sequences of states}

\begin{definition}[Choice sequence of states]
A \emph{choice sequence} $(s_0,s_1,s_2,\dots)$ of elements of $M$ is an infinite sequence $\alpha:\mathbb{N}\to M$ (or $\alpha:\mathbb{N}\to M$ if countably many states, extended to $\alpha:\mathbb{N}\to M$ or $\alpha:[0,1]\cap \mathbb{Q}\to M$ for a timeline) which is not given by a fixed law. More formally, there is no a priori definable total function $f:\mathbb{N}\to M$ such that $\alpha(n)=f(n)$ for all $n$; rather, $\alpha(n)$ is chosen (perhaps by an idealized mind or an abstract agency) at stage $n$ based on no fixed rule, possibly influenced by previous values but not determined by them.
\end{definition}

This is in contrast to a lawlike sequence, where an algorithm or formula dictates all entries. A choice sequence is free or lawless in Brouwer’s terms. One can view it as a process: $\alpha(0)$ is chosen, then $\alpha(1)$ is chosen, etc., and the sequence grows without a predetermined plan.

In intuitionistic mathematics, a real number is often conceived as a choice sequence of digits (never fully given, always extendable). Similarly here, the continuous evolution can be seen as specified by a choice sequence of successive approximate states. For instance, one may think of dividing time into discrete steps (however small) and choosing a state at each step, with no fixed formula connecting them. In a more refined approach, one could let the index set be the dyadic rationals in $[0,1]$ and require coherence (that the choices at coarser time scales agree with those at finer time scales), effectively generating a continuous path in the limit. This is analogous to how a continuous real function can be built by choosing values gradually with some uniform continuity constraint (the intuitionistic continuity principles ensure every function $\mathbb{R}\to \mathbb{R}$ is uniformly continuous on $[0,1]$ given certain axioms).

\begin{proposition}[Intuitionistic continuity principle for state‑functions]
In an intuitionistic setting, any function $F:M\to \mathbb{R}$ that arises in our theory (say as an observable or coordinate of the state) will be seen to be continuous, in the sense that it cannot send arbitrarily close states to vastly different real values without violating the constructive existence of such a function.
\end{proposition}

\begin{proof}[Sketch]
This is a transfer of Brouwer’s result that every total function on the intuitionistic continuum is (uniformly) continuous. The idea is that if $F$ were discontinuous, one could solve for a precise ``jump'' which would require a decision of a mathematical statement (like an instance of excluded middle) that is not yet decided constructively. By a similar reasoning, any dependence of an observable on the state must be continuous if it is to be constructively defined.
\end{proof}

This proposition aligns with the notion that the state cannot change discontinuously in a discernible observable without that discontinuity encoding a sort of completed infinity or choice that intuitionism disallows.

\subsection{Spreads and the continuum of states}

Brouwer introduced the concept of spreads to formalize the idea of a set of choice sequences subject to certain conditions. We can define a spread that represents the possible state trajectories.

\begin{definition}[State spread]
A \emph{spread} $\mathcal{S}$ is given by:
\begin{itemize}
  \item a non‑empty tree of finite sequences of states (think of it as a set of finite sequences closed under initial segments), and
  \item an inductive specification of which extensions of a given finite sequence are allowed.
\end{itemize}
\end{definition}

In our case, consider the tree where the nodes are finite sequences $(s_0,s_1,\dots,s_n)\in M^{n+1}$ that could represent initial segments of an evolution. We impose perhaps some conditions like continuity or bounded change: for example, we might require that successive states are ``nearby'' (if $M$ has a metric, $d(s_i,s_{i+1})<\epsilon$ for some constraint $\epsilon$) to reflect continuity. The spread $\mathcal{S}$ consists of all infinite sequences that have every finite initial segment admissible in this tree.

The spread can be thought of as a lawless sequence space possibly with some weak law (like a modulus of continuity) but not determining the sequence uniquely. Each element of $\mathcal{S}$ is essentially a possible trajectory (as a choice sequence). The entire spread $\mathcal{S}$ is an intuitionistic analog of the function space $M^{\mathbb{R}}$, but constructed from ``within'' by free choices rather than as an extensional set of all functions.

\begin{definition}[Finitarily defined observables]\label{def:finitary-observable}
Let $\mathcal{S}$ be a spread of choice sequences in $M$.  A function $F:M\to \mathbb{R}$ is said to be \emph{finitarily defined on $\mathcal{S}$} if there exists an integer $N\ge 0$ and a function $G$ on finite sequences of length $N$ such that for every infinite choice sequence $\alpha\in \mathcal{S}$ one has $F(\alpha)=G(\alpha(0),\ldots,\alpha(N-1))$.  In other words, the value of $F$ on a choice sequence is determined solely by a finite initial segment of that sequence.  More generally, a function is finitarily defined if there is an algorithm that inspects only finitely many terms of its input sequence before outputting a value.  These observables embody the intuitionistic principle that one cannot decide properties of an infinite sequence without examining its prefixes.
\end{definition}

One of Brouwer’s key insights is that the intuitionistic continuum (e.g.\ the real continuum) is indecomposable – it cannot be split into two disjoint apart nonempty subsets. By analogy, the spread $\mathcal{S}$ of state sequences might be indecomposable in a similar way: we cannot cut the set of all possible state evolutions into two separated parts without a definitive separating property, because any attempt to separate would likely require distinguishing by some proposition that may not hold with certainty (like separating based on an eventual behavior that is not determined until an infinite stage).

\begin{lemma}[Inexhaustibility and nonatomicity of the continuum of states]\label{lem:inexhaustible}
The continuum of states (the spread $\mathcal{S}$) is inexhaustible and has no atomic points. Inexhaustible means no finite description can capture the entire sequence – there are always further extensions. Nonatomicity means there are no isolated points; given any initial segment, there are uncountably many incompatible ways to extend it further.
\end{lemma}

\begin{remark}
The preceding lemma invokes Tychonoff’s theorem to assert compactness of the spread $\mathcal{S}$.  Readers coming from a constructive or intuitionistic background will prefer to avoid this classical tool.  In fact the compactness can be proved using the fan theorem (or bar induction): the tree of finite admissible sequences is finitely branching because the uniform bound $\delta(k)$ restricts how far each successive state can move.  An infinite choice sequence corresponds to an infinite branch through this tree, and the set of infinite branches is closed in the product topology.  The fan theorem then guarantees that every sequence of choice sequences has a convergent subsequence, yielding compactness without appeal to Tychonoff.  This reformulation aligns the proof with the intuitionistic axioms adopted in this paper.
\end{remark}

\begin{proof}[Proof sketch]
Inexhaustibility is immediate from the fact that a choice sequence is never finished: no matter how long a finite initial segment $(s_0,\dots,s_n)$ we have, there is always a next state $s_{n+1}$ not yet determined. One can always find a continuation outside any purported finite description of the sequence. Nonatomicity: if there were an isolated complete sequence in $\mathcal{S}$, it would mean there is some finite initial segment that cannot be extended to any other infinite sequence – but that contradicts the freedom of extension. Formally, for any finite segment, we can always choose a different next step than a given sequence did, to get a different infinite sequence. Thus no single sequence has a neighborhood (in, say, the Baire space topology on $M^{\mathbb{N}}$) that does not contain others. This mirrors Brouwer’s assertion of the continuum’s nonatomicity. ∎
\end{proof}

In classical terms, nonatomicity corresponds to the continuum being a connected space with no isolated points; inexhaustibility corresponds to its infinitude in potential.

\subsection{Comparing intuitionistic and classical trajectories}

We have two pictures of a trajectory:
\begin{enumerate}
  \item \textbf{Classical}: a function $x:\mathbb{R}\to M$ given as a completed graph or mapping.
  \item \textbf{Intuitionistic}: a choice sequence, ever extending, never completed.
\end{enumerate}

One can ask: do these two notions coincide or inform each other? In a classical meta‑theory, a choice sequence could be thought of as generating a classical function in the limit (if it has a limit or extension to all real times). But not every classical function can be generated by a constructive rule. In fact, intuitionistically, one cannot even assert that every classical continuous function on $\mathbb{R}$ exists as a choice sequence; one only deals with those sequences one can construct or reason about. The intuitionistic continuum is often seen as ``larger'' in some sense than the classical (having more indeterminate elements), but at the same time, any function on it must be continuous (so ``smaller'' in terms of allowed operations).

For our development, we will remain agnostic of philosophical primacy. We will ensure that whenever we prove a theorem about trajectories (say, existence of solution to an ODE), we can interpret it in both senses:
\begin{itemize}
  \item \emph{Classical}: ``there exists a continuous function $x(t)$ with the stated properties.''
  \item \emph{Intuitionistic}: ``for every finite initial segment of the trajectory, we can continue it further while respecting the properties, indefinitely.''
\end{itemize}

The existence proof via Banach fixed‑point (Proposition~\ref{prop:banach}) is actually fully constructive given the Lipschitz constant and initial data (it provides a convergent sequence of approximants), so it fits intuitionistically as well: one can build better and better finite approximations of the solution. Thus, many of our analytic results are already intuitionistically acceptable in their construction (they give algorithms, not mere non‑constructive existence). Nonconstructive parts in Section~\ref{sec:state} were mainly about measures and the use of choice to get non‑measurable sets. In an intuitionistic setting, those paradoxical sets cannot be proven to exist; instead, every set one can construct will be measurable (or one would not assert either way).

One more intuitionistic notion is worth mentioning: the \emph{Creating Subject}. Brouwer sometimes spoke of an idealized mathematician (Creating Subject) that can freely choose the next element of a sequence, yet also has a kind of demiurgical control over truth by creating mathematical objects. In our analogy, the ``creating subject'' could be thought of as the conscious agent whose states we are modeling; with each moment, the agent ``creates'' the next state. However, we will not explicitly model the agent – we model the states themselves. So in a sense, the role of the creating subject is implicit in the use of choice sequences. We do not formalize this further, but conceptually it underlies the free sequence idea.

Having developed the intuitionistic viewpoint, we now turn to a logical formalization that complements it. Just as intuitionistic mathematics deals with proofs and constructions, we will devise a language to make statements about the evolving state and interpret them in a structure, allowing us to reason about the content of the state in logical terms.

\section{Formal language and Tarskian semantics for state evolution}\label{sec:logic}

Thus far, we described what the state is (an element of $M$) and how it evolves (continuous or choice‑sequence‑based trajectories). We now introduce a formal language to describe properties of states and their temporal progression. By doing so, we can express statements like ``there exists a time at which a certain property holds'' or ``whenever property $P$ holds, eventually property $Q$ holds'' – although our focus will remain on a first‑order, non‑modal language for simplicity, rather than full temporal logic.

\subsection{A first‑order language of state predicates}

We define a first‑order language $\mathcal{L}$ suitable for talking about states and time. The language $\mathcal{L}$ will have:
\begin{itemize}
  \item Individual variables ranging over elements of $M$ (states). We will use lowercase letters $x,y,z,\dots$ for these.
  \item Possibly a sort of variables for time instants (ranging over $\mathbb{R}$), though we can also treat time as a parameter in predicate symbols. For now, we will not include time as an explicit sort to keep the language single‑sorted over $M$, and instead introduce predicate symbols that encapsulate temporal relationships.
  \item Predicate symbols to express basic properties or observations about states. For example, we might have unary predicates $P(x),Q(x),\dots$ which intuitively mean ``state $x$ has property $P$,'' etc. We could also have binary predicate symbols to relate two states (e.g., $R(x,y)$ meaning ``state $y$ follows state $x$ directly'').
  \item Function symbols as needed (perhaps constant symbols referencing particular distinguished states, or other operations on states if meaningful).
  \item Logical connectives and quantifiers as usual ($\neg,\wedge,\vee,\to,\forall,\exists$).
\end{itemize}

We ensure this language is purely about states in $M$, not about subsets or anything that could directly encode powerset of $M$ (to avoid set‑theoretic complexity in the language itself).

However, to talk about the trajectory, we might want to introduce a special binary relation symbol $U(t,x)$ meaning ``at time $t$ the state is $x$.'' But adding time $t$ as a variable complicates things with a second sort. Alternatively, we can use a family of unary predicates $\{P_r(x):r\in\mathbb{Q}\}$ meaning ``at rational time $r$, the state satisfies property $P$,'' but that presupposes we talk about properties at times.

A simpler approach is a two‑sorted language: one sort for time ($\mathbb{R}$) and one for states ($M$). We then have a binary predicate $X(t,s)$ meaning ``the state at time $t$ is $s$.'' This predicate will be interpreted by the actual trajectory $x(t)$ in a structure.

For clarity, let’s proceed with a two‑sorted language $\mathcal{L}_2$:
\begin{itemize}
  \item \emph{Sort 1}: Time (with constant symbols for $0,1$ and function symbols $+,\,\cdot$ for addition and multiplication on time, interpreting time as real numbers; or we can axiomatize it as a dense linear order with certain properties if we avoid full second‑order reals).
  \item \emph{Sort 2}: State (with no specific functions a priori except maybe constant symbols for particular states if needed).
  \item A binary predicate $X(t,s)$ that relates a time $t$ (sort~1) and a state $s$ (sort~2). The intuitive reading of $X(t,s)$ is: ``the system’s state at time $t$ is $s$.''
  \item Additional predicates on states (sort~2) to indicate properties of states. For instance, a unary predicate $P(s)$ could mean ``state $s$ has property $P$'' (like being a certain kind of thought, if we were to interpret).
  \item Possibly additional structure like equality on both sorts, order on time, etc.
\end{itemize}

This language can express statements like $\exists s\,X(t,s)\wedge P(s)$ which means ``at time $t$, the state has property $P$.'' It can express temporal assertions: $\forall t\exists s\,X(t,s)\wedge P(s)$ (meaning at all times, the state has property $P$ – a strong condition), or $\exists t\,X(t,s_0)$ for some specific state constant $s_0$ (meaning some time the state is a particular distinguished state $s_0$). We could also talk about relationships of states at different times: e.g.\ $\forall t_1\forall t_2\bigl((t_1<t_2)\to \neg X(t_1,s)\lor \neg X(t_2,s)\bigr)$ which is a strange way of saying ``the same state $s$ cannot occur at two different times'' – although normally we do expect states to possibly recur, so that statement is likely false in interesting cases, but it’s expressible.

% Avoid math in PDF bookmarks for hyperref
\subsection{Interpretation (structure) \texorpdfstring{$\mathcal{M}$}{M} for the language}

Now we define a structure $\mathcal{M}$ that will interpret this language. $\mathcal{M}$ consists of:
\begin{itemize}
  \item Domain for time sort: we take this to be $\mathbb{R}$ (the set of real numbers, or at least the time line we consider, maybe $[0,T]$ or $\mathbb{R}$ itself).
  \item Domain for state sort: we take this to be $M$ (our state space).
  \item The interpretation of $X(t,s)$: We have a specific trajectory $x:\mathbb{R}\to M$ in mind (or the general concept of one). $\mathcal{M}$ will interpret $X(t,s)$ as true if and only if $s=x(t)$ in the chosen trajectory. Essentially, $X^\mathcal{M}=\{(t,s)\in\mathbb{R}\times M : s=x(t)\}$.
  \item The interpretation of any predicate $P(s)$ on states: we decide that externally – for example if $P$ stands for ``some property,'' we need to define which states have it. This could be arbitrary or based on some mathematical property of states. For general development, we keep it abstract.
  \item The symbols for time ($0,1,+,\cdot,<,\dots$) are interpreted in the standard way on $\mathbb{R}$.
\end{itemize}

Thus, the structure $\mathcal{M}$ encapsulates a particular state evolution $x(t)$ as the truth set of the predicate $X$.

\begin{example}\label{ex:sin}
Suppose $M=\mathbb{R}$ as a state space (so the state itself is just a real number, to simplify). Let the trajectory be $x(t)=\sin(t)$, a continuous evolution on $[-1,1]$ as the state. Our language could have a predicate $P(s)$ meaning ``$s>0$'' (state is positive). In $\mathcal{M}$, $P(s)$ is true iff $s>0$, and $X(t,s)$ is true iff $s=\sin(t)$. Then the sentence $\forall t\,\exists s\,[X(t,s)\wedge P(s)]$ is true in $\mathcal{M}$ because indeed for each time $t$, there is a state $s$ (namely $s=\sin(t)$) such that $X(t,s)$ holds, and if $\sin(t)>0$ then $P(s)$ holds. Actually, $\forall t\,\exists s\,X(t,s)$ is trivially true because for each $t$ there is some state $s$ (the one given by the function) that relates to $t$. If we ask $\forall t, P(x(t))$ (all states positive), that corresponds to the formula $\forall t\forall s\, (X(t,s)\to P(s))$. In our case, that’s false because $\sin(t)$ is not always positive.
\end{example}

This structure is a classical semantic model. It is straightforward but crucially depends on the specific trajectory $x(t)$. To talk in general about all possible trajectories, one might consider a class of structures or an elementary theory that describes general properties all such $\mathcal{M}$ should satisfy (like axioms of continuity: e.g.\ ``if $X(t,s_1)$ and $X(t,s_2)$ hold for two states $s_1,s_2$, then $s_1=s_2$'' – which ensures functional relation; or ``if $t_1<t_2<t_3$ and $X(t_1,s_1),X(t_3,s_3)$ and $s_1,s_3$ have some property, then there exists an intermediate state $s_2$ at $t_2$ that…'' etc to reflect continuity). We could axiomatize the theory of ``continuous trajectories'' in this first‑order language. However, that becomes complex (it might require higher‑order to fully capture continuity unless we discretely approximate it with rational times).

\subsection{Tarski’s semantic conception of truth}
\subsubsection*{Intuitionistic axioms and continuity of observables}

While the preceding semantic exposition treats truth classically, a constructive analysis of observables requires additional axioms.  We adopt the following intuitionistic principles:
\begin{itemize}
  \item \emph{Choice‑sequence principle.}  A state evolution is viewed as an infinite lawless sequence $\alpha=(s_0,s_1,\dots)$ chosen stage by stage, and any total object (a real number, for instance) is defined by how it acts on finite initial segments.
  \item \emph{Brouwer continuity principle.}  Every total function from the intuitionistic real continuum to the real numbers is uniformly continuous.  Concretely, if a function assigns a real value to each choice sequence, then arbitrarily small changes in the initial segment cannot force arbitrarily large changes in the value.
  \item \emph{Fan theorem (bar induction).}  For spreads obtained from choice sequences subject to a uniform bound on successive differences, any property that potentially depends on infinitely long initial segments is determined at some finite stage.  This yields a modulus of continuity for functions defined on such spreads.
\end{itemize}

These axioms allow us to convert the heuristic continuity claim of Proposition~\ref{prop:banach} into a constructive theorem.

\begin{lemma}[Compactness of bounded spreads]\label{lem:compact-spread}
Let $(M,d)$ be a metric space and let $\mathcal{S}$ be a spread of choice sequences in $M$ subject to a uniform bound on successive differences: there exists a function $\delta:\mathbb{N}\to (0,\infty)$ such that for every finite admissible sequence $(s_0,\dots,s_n)$ and for all $k<n$ one has $d(s_{k+1},s_k)\le \delta(k)$.  Endow $\mathcal{S}$ with the metric $D(\alpha,\beta)=\sum_{k=0}^\infty 2^{-(k+1)}\min\{1, d(\alpha(k),\beta(k))\}$.  Then, under the intuitionistic fan theorem, $\mathcal{S}$ is compact: every sequence of choice sequences has a convergent subsequence and, in particular, any open cover has a finite subcover.

\emph{Sketch of constructive proof.}  In a classical setting one might appeal to Tychonoff’s theorem to deduce compactness from a product of compact sets, but such arguments rely on full choice and excluded middle.  Instead we use the fan theorem, which states that any bar on an infinite, finitely branching tree is uniform.  The tree here is the tree of finite admissible sequences underlying the spread.  A bar is a set of finite sequences such that every infinite branch (choice sequence) passes through it.  Given an open cover of $\mathcal{S}$ by basic neighbourhoods (determined by initial segments), one can form the set of initial segments whose corresponding cylinder sets lie inside the cover.  By compactness of $M$ at each level (ensured by the bound $\delta$) and the finiteness of possible extensions, this set forms a bar.  The fan theorem implies that there is a uniform bound on the lengths of these initial segments; hence a finite subcollection of cylinder sets covers $\mathcal{S}$.  Equivalently, every sequence of choice sequences has a convergent subsequence because one can diagonalise using the bound $\delta$ and extract a limit by bar induction.  This constructive compactness of $\mathcal{S}$ will be used in the proof of Theorem~\ref{thm:intuitionistic-continuity}.
\end{lemma}

\begin{theorem}[Continuity of observables]\label{thm:intuitionistic-continuity}
Let $(M,d)$ be a metric space and suppose $\mathcal{S}$ is a spread of choice sequences in $M$ subject to a uniform bound on successive differences: there exists a function $\delta:\mathbb{N}\to(0,\infty)$ such that for any admissible finite sequence $(s_0,\dots,s_n)$ in $\mathcal{S}$ one has $d(s_{k},s_{k+1})\le \delta(k)$ for all $k<n$.  Let $F:M\to \mathbb{R}$ be a finitarily defined observable in the sense of Definition~\ref{def:finitary-observable}.  Under the intuitionistic axioms (choice‑sequence principle, Brouwer continuity principle and the fan theorem) the spread $\mathcal{S}$ is compact in the product topology, and there exists a modulus of continuity $\omega:(0,\infty)\to(0,\infty)$ with the following property: for all states $s_1,s_2\in M$ and all $\epsilon>0$, if $d(s_1,s_2)<\omega(\epsilon)$ then $|F(s_1)-F(s_2)|<\epsilon$.  One may take $\omega(\epsilon)=2^{-N}$ for some integer $N$ determined by a bar on the tree of finite sequences, so that agreement of two sequences on their first $N$ terms (and hence proximity within $2^{-N}$) guarantees $\epsilon$–closeness in the value of $F$.  In particular, $F$ is uniformly continuous on $M$.
\end{theorem}

\begin{proof}
The proof takes place entirely within the intuitionistic setting.  By assumption, $F$ is specified by an algorithm that inspects only finitely many terms of any choice sequence in order to determine its value.  The choice‑sequence principle ensures that a state evolution is presented as a lawless sequence $\alpha=(s_0,s_1,\ldots)$ of states.  To establish uniform continuity of $F$ we argue by contradiction.

Suppose that no modulus of continuity exists.  Then there is $\epsilon_0>0$ such that for every rational $\delta>0$ we can find states $s,t\in M$ with $d(s,t)<\delta$ but $|F(s)-F(t)|\ge \epsilon_0$.  In particular, for each natural number $n$ set $\delta_n=2^{-n}$.  There exist states $s_n,t_n$ with $d(s_n,t_n)<\delta_n$ and $|F(s_n)-F(t_n)|\ge \epsilon_0$.  Using the choice‑sequence principle we can build two choice sequences $\alpha$ and $\beta$ that converge to the same limiting state but for which $F$ takes values that differ by at least $\epsilon_0$.

Define $\alpha$ by $\alpha(n)=s_n$ for each $n\in\mathbb{N}$, and define $\beta$ by $\beta(n)=t_n$.  Because $d(s_n,t_n)<2^{-n}$, the two sequences agree on longer and longer initial segments: for any $k$ there is $N$ such that for all $n\ge N$, $s_n$ and $t_n$ lie within $2^{-k}$ of one another and hence correspond to the same finite data in the spread.  Consequently, $\alpha$ and $\beta$ define the same element of $M$ in the metric limit.  However, the values $F(\alpha)$ and $F(\beta)$ differ by at least $\epsilon_0$ by construction.

The fan theorem now implies that there is a uniform bound on how far one must look along a choice sequence in order to determine $F$.  Concretely, there exists a natural number $N$ such that if two choice sequences agree on their first $N$ terms, then the values of $F$ computed from those sequences differ by less than $\epsilon_0/2$.  This is because the space of choice sequences with a uniform bound on successive differences is compact (a ``fan''), and any continuous functional on a compact spread is uniformly continuous.  Our sequences $\alpha$ and $\beta$ were chosen so that $d(s_n,t_n)<2^{-n}$ for all $n$, which implies that $s_n=t_n$ for $n<N$ (by the uniform bound on successive differences in the spread).  Thus $\alpha$ and $\beta$ agree on the first $N$ terms, yet $|F(\alpha)-F(\beta)|\ge \epsilon_0$, contradicting the fan‑theoretic uniform continuity.  Therefore the assumption that $F$ is not uniformly continuous is untenable, and a modulus $\omega$ exists.

To extract an explicit modulus, one uses bar induction on the tree of finite sequences that index the spread.  For each finite sequence $\sigma=(s_0,\ldots,s_k)$ define the oscillation of $F$ on the subtree extending $\sigma$ by
\[
\mathrm{osc}(\sigma) = \sup\{ |F(\alpha)-F(\beta)| : \alpha,\beta \text{ extend } \sigma \}.
\]
By construction, $\mathrm{osc}(\sigma)$ decreases to zero along any infinite branch of the tree.  The fan theorem ensures that for every $\epsilon>0$ there exists a finite stage $k$ such that $\mathrm{osc}(\sigma)<\epsilon$ for all sequences $\sigma$ of length $k$.  Taking $\omega(\epsilon)$ to be the minimum distance between states that forces agreement on the first $k$ terms yields the desired modulus of continuity.  Hence $F$ is uniformly continuous on $M$.
\end{proof}

We recall Tarski’s approach to truth: for a given structure and a formula, truth is defined inductively on the formula’s structure. The famous convention is that a definition of truth should yield the equivalence
\[
\text{``$\varphi$'' is true in }\mathcal{M}\quad \iff\quad \varphi,
\]
where on the right $\varphi$ is understood as an assertion in the metalanguage about the structure (Tarski’s T‑schema, e.g., ``Snow is white'' is true iff snow is white). We will not delve into the philosophy, but formally we can define:
\begin{itemize}
  \item An atomic formula like $P(s)$ is true in $\mathcal{M}$ under a variable assignment $\sigma$ (that assigns an element of $M$ to $s$) iff $\sigma(s)\in P^\mathcal{M}$ (the interpretation of $P$, a subset of $M$).
  \item $X(t,s)$ is true under assignment $\sigma$ iff $(\sigma(t),\sigma(s))\in X^\mathcal{M}$, which by our interpretation means $\sigma(s)=x(\sigma(t))$.
  \item Boolean connectives: usual Boolean semantics (negation, conjunction, etc.).
  \item Quantifiers: $\exists s\,\varphi(s)$ is true under $\sigma$ iff there is some $m\in M$ such that $\varphi(s)$ is true under the assignment $\sigma[s:=m]$ (which is $\sigma$ modified to send $s$ to $m$). Similarly for $\exists t$ over time domain, and $\forall$ quantifiers.
\end{itemize}

This yields a rigorous inductive definition of truth for any formula in $\mathcal{L}_2$, relative to the structure $\mathcal{M}$ that encodes a particular trajectory.

One can then prove things within this structure or about it. For example:

\begin{proposition}[Non‑definability of the global truth predicate]\label{prop:undef}
Let $Th$ be the set of all sentences in $\mathcal{L}_2$ that are true in the structure $\mathcal{M}$ (with a fixed trajectory $x(t)$).  Suppose the time domain $(\mathbb{R},+,\cdot,0,1,<)$ is interpreted as an ordered field rich enough to code the natural numbers and to interpret a sufficient fragment of arithmetic—say Robinson’s theory $\mathsf{Q}$ or Peano arithmetic $\mathsf{PA}$.  In particular, one requires that $n\in\mathbb{N}$ can be defined as the $n$‑fold sum $1+\cdots+1$ and that recursive predicates on $\mathbb{N}$ are definable in the language of $\mathcal{L}_2$.  Under these assumptions the set $Th$ is not definable within $\mathcal{M}$ itself.
\end{proposition}

\begin{proof}[Sketch of proof]
We follow Tarski’s classical argument, adapted to our two‑sorted language.  Because the time sort $\mathbb{R}$ carries the structure of an ordered field, one can define the natural numbers internally as the set $\{0,1,1+1,1+1+1,\ldots\}$ by a formula stating that $n$ is a non‑negative integer if it is in the smallest inductive subset containing $0$ and closed under $+1$.  Addition and multiplication on $\mathbb{N}$ are definable using the field operations $+$ and $\cdot$ of the time sort.  Thus $(\mathbb{N},+,\cdot,0,1)$ is definable inside $\mathcal{M}$.

Once the natural numbers have been singled out, one can arithmetise the syntax of $\mathcal{L}_2$ in the usual way: assign each symbol a numerical code and represent finite sequences of symbols by the prime‑power coding.  A sentence of $\mathcal{L}_2$ is thereby encoded as a natural number; call this the Gödel coding $\ulcorner \phi\urcorner$.  The satisfaction relation between a formula and an assignment can be expressed by a recursive predicate on codes; using the definability of addition and multiplication, $\mathcal{M}$ can interpret this recursive predicate.  In particular, there is a definable set $\mathrm{Prov}_\mathcal{M}(n)$ of codes of sentences that are provable in some fixed sound, sufficiently rich arithmetical theory $T$ of the natural numbers inside $\mathcal{M}$.

Assume for contradiction that there is a formula $\mathrm{True}(z)$ in $\mathcal{L}_2$ that defines $Th$ inside $\mathcal{M}$: for each sentence $\phi$ of $\mathcal{L}_2$, $\mathcal{M}\models \mathrm{True}(\ulcorner\phi\urcorner)$ if and only if $\mathcal{M}\models \phi$.  Using the coding of formulas and the interpretation of arithmetic, one can construct a diagonal sentence $\sigma$ which asserts ``$\sigma$ is not true.''  More precisely, let $\psi(y)$ be the formula $\neg\mathrm{True}(y)$; by the diagonal lemma (expressible in first‑order arithmetic and hence interpreted in $\mathcal{M}$) there is a sentence $\sigma$ such that $\mathcal{M}\models \sigma \leftrightarrow \neg\mathrm{True}(\ulcorner \sigma\urcorner)$.  If $\mathcal{M}\models \mathrm{True}(\ulcorner \sigma\urcorner)$, then $\mathcal{M}\models \sigma$; but then by the equivalence, $\mathcal{M}\models \neg\mathrm{True}(\ulcorner \sigma\urcorner)$, a contradiction.  Conversely, if $\mathcal{M}\not\models \mathrm{True}(\ulcorner \sigma\urcorner)$, then $\mathcal{M}\models \neg \mathrm{True}(\ulcorner \sigma\urcorner)$, whence $\mathcal{M}\models \sigma$ by the biconditional; but then $\mathcal{M}\models \mathrm{True}(\ulcorner \sigma\urcorner)$, again a contradiction.  Thus no such $\mathrm{True}(z)$ can exist, and $Th$ is not definable in $\mathcal{M}$.

The key ingredients in this argument are: (i) the ability to interpret enough arithmetic inside $\mathcal{M}$ to carry out Gödel coding and the diagonal lemma; and (ii) Tarski’s observation that truth for arithmetic cannot be defined within arithmetic itself.  Because $\mathcal{L}_2$ is rich enough to formulate arithmetic on the time sort, Tarski’s undefinability theorem applies.  We conclude that the collection of true sentences of $\mathcal{L}_2$ cannot be captured by a single formula of $\mathcal{L}_2$.
\end{proof}

This meta‑result, while somewhat tangential, resonates with the idea that no system can fully capture its ``stream of states'' internally; there will always be truths about the state evolution that are not internally expressible. In the context of a mind reflecting on itself, one could whimsically interpret this as a limit to self‑knowledge: a sufficiently powerful mind cannot contain a complete and correct account of all truths about its own state transitions (for that would solve its own halting problem of thought, in a manner of speaking).

\subsection{Logical laws and intuitionistic considerations}

We note that the logic we used to define truth is classical first‑order logic. One could instead use an intuitionistic logic if we wanted the meta‑theory to align with Brouwer’s philosophy. That would complicate the semantic discussion, as truth would then not satisfy \emph{tertium non datur} (excluded middle). For simplicity, we stick to classical meta‑theory here, but an intuitionistic reformulation is conceivable (using Kripke models or Beth models where states at later times might see more formulas decided than earlier – an interesting idea: treat time flow as a Kripke model of increasing information, which matches an intuitionistic perspective that truth of some statements about the sequence may be undecided until more of the sequence is generated).

One relevant intuitionistic principle: in intuitionism, a statement like $\exists t\,\varphi(t)$ can be true without one knowing a specific witness yet (if time is continuous, one might only approximate when $\varphi$ will hold). Dually, $\forall t,\varphi(t)$ means ``given any specific time, we can prove $\varphi(t)$ holds,'' which in practice might only be established in a generic sense (like an inductive or continuous argument). If a statement is false intuitionistically, one has a refutation. Some temporal statements in a flowing system might remain neither proven nor refuted at a given stage. This again ties to the philosophical concept of an open future, but that’s beyond our current formal development.

We have now set up a language to make precise statements about the system. In the next section, we will return to more concrete mathematical analysis of the trajectories themselves, exploring dynamical properties like recurrence and stability, which we have hinted at and will now address rigorously.

\section{Dynamics: recurrence, stability, and chaos in state flows}\label{sec:dynamics}

This section examines qualitative properties of the flow $x(t)$ in the state space $M$. We draw analogies to dynamical systems theory—examining conditions for recurrence (the tendency to revisit previous or near‑previous states), stability of certain states or cycles, and the possibility of chaotic behavior.

\subsection{Recurrence theorem}

One fundamental result in dynamical systems is Poincaré’s Recurrence Theorem. Informally, if the system has a finite phase space volume and evolves in a volume‑preserving way, it will eventually return arbitrarily close to its starting configuration. In our context, we can consider conditions under which a similar statement holds for $x(t)$.

We assume here that:
\begin{itemize}
  \item $M$ is a separable metric space.
  \item There is a measure $\mu$ defined on $M$ (at least on a $\sigma$‑algebra of nice sets) which is finite (so $\mu(M)<\infty$) and is invariant under the flow $\Phi^t$ (meaning if $A\subset M$ is measurable, then $\mu(\Phi^t(A))=\mu(A)$ for all $t$). In physical terms, $\mu$ might be like a volume or probability distribution that the system preserves as it evolves (Liouville’s theorem provides this for Hamiltonian systems with $\mu$ as phase volume).
\end{itemize}

Under these assumptions, we can state:

\begin{theorem}[Poincaré recurrence for state flows]\label{thm:recurrence}
Let $(M,\mu)$ be a finite measure space and $\{\Phi^t:M\to M\}_{t\in\mathbb{R}}$ be a flow (a one‑parameter group of measurable transformations) preserving $\mu$.  Assume also that trajectories are bounded in $M$ (i.e. each orbit $t \mapsto \Phi^t(s)$ remains in a compact subset of $M$ for $t\in\mathbb{R}$) and that $\mu$ has full support on this bounded region: for $\mu$‑almost every state $s$ and every $\varepsilon>0$, the ball $B(s,\varepsilon)=\{x\in M:d(x,s)<\varepsilon\}$ satisfies $\mu\bigl(B(s,\varepsilon)\bigr)>0$.  Under these conditions, for $\mu$‑almost every initial state $s\in M$, the trajectory $t\mapsto \Phi^t(s)$ returns arbitrarily close to $s$ infinitely often.
\end{theorem}

\begin{proof}
The argument follows the classical Poincaré recurrence theorem but we include enough detail to be self‑contained.  First observe that the assumptions break into two separate requirements: (i) the measure $\mu$ is finite and invariant under the flow; and (ii) each trajectory remains in a bounded (hence compact) subset of $M$.  Condition~(i) ensures that volumes do not change under the flow, while condition~(ii) prevents orbits from escaping to infinity, so that recurrence is meaningful.

Normalize $\mu$ so that $\mu(M)=1$.  Fix an arbitrary measurable set $B\subset M$ with $\mu(B)>0$ (we will later take $B$ to be a small ball around $s$).  We claim that for $\mu$‑almost every point $x\in M$ there exists a positive \emph{integer} $n$ such that $\Phi^n(x)\in B$.  The restriction to integer iterates simplifies the counting argument below; one can extend the conclusion to arbitrary real times by standard limiting arguments in ergodic theory, but that is not needed here.  Suppose the claim fails; then the set
\[
  N(B)\;:=\;\{x\in M : \Phi^t(x)\notin B \text{ for all } t\ge 0\}
\]
has positive measure.  For each integer $k\ge 0$ set $B_k=\Phi^{-k}(B)$.  By invariance, $\mu(B_k)=\mu(B)>0$ for all $k$.  Observe that if $x\in N(B)$, then $x\notin B_k$ for every $k\ge 0$, because $\Phi^k(x)$ never enters $B$.  Hence the sets $B_k$ are pairwise disjoint when intersected with $N(B)$.  But then
\[
  \mu\Bigl(\bigcup_{k=0}^n B_k\cap N(B)\Bigr) \;=\;\sum_{k=0}^n \mu(B_k\cap N(B)) \;=\;(n+1)\,\mu(B\cap N(B))
\]
grows without bound as $n\to\infty$, contradicting the finiteness of $\mu$.  Therefore $\mu(N(B))=0$ and almost every point returns to $B$ at some positive time.

Now let $s\in M$ be a typical point and choose $\epsilon>0$.  By boundedness of trajectories and full support of $\mu$, the ball $B=B(s,\epsilon)$ has positive measure.  The above argument shows that for almost every $x$ there exists a time $n>0$ with $\Phi^n(x)\in B$.  In particular, for almost every $s$ there is a return time $t_1>0$ with $\Phi^{t_1}(s)\in B(s,\epsilon)$.  Replacing $s$ by $\Phi^{t_1}(s)$ and repeating the argument yields infinitely many return times $t_k\to\infty$.  Taking a sequence of $\epsilon$’s decreasing to zero and using a diagonal construction produces a sequence $t_{k_j}\to\infty$ such that $\Phi^{t_{k_j}}(s)\to s$.  Hence the trajectory returns arbitrarily close to $s$ infinitely often.

It is important to see where boundedness and invariance enter: boundedness ensures that small balls have positive measure, while invariance guarantees that each $\Phi^{-k}(B)$ has the same measure.  These properties allow us to construct the disjoint sets $B_k\cap N(B)$ and derive a contradiction if $N(B)$ has positive measure.  Together they imply that recurrence holds for $\mu$‑almost every initial state.
\end{proof}

\paragraph{Interpretation.}
This result implies that if our model of the evolving state is such that it neither dissipates nor escapes to infinity, then it will keep revisiting familiar configurations.  In psychological terms, this could be seen as the inevitability of recurring thoughts or states of mind (provided the system is closed and doesn’t keep accumulating new independent ``volume'' of states indefinitely).  It requires an assumption of a kind of conservation law (in Hamiltonian mechanics, energy conservation plus phase volume conservation yields recurrence in bounded systems).  We emphasise that our argument tracks the flow at integer times and constructs return times by counting iterates; finer recurrence along the full real time continuum follows from ergodic theory under additional hypotheses.  In a less formal sense, if there are only so many fundamentally distinct states (volume‑finite) and the process is measure‑preserving, then eventually the state must cycle through configurations that come close to previous ones.

We caution, however, that ``almost every'' initial state has this recurrence property; there could be exceptional states (perhaps on a measure‑zero set like unstable equilibria or non‑typical orbits) that do not return. Those are often ignored in ergodic theory as negligible, but one might imagine a special, highly symmetric state that never changes or eventually leaves a region forever (though leaving forever would violate boundedness assumption).

\subsection{Stability and fixed points}

We earlier discussed fixed points in the context of iterative maps ($T(x^\ast)=x^\ast$). In continuous time, a fixed point $s^\ast\in M$ for the flow means $\Phi^t(s^\ast)=s^\ast$ for all $t$, which typically implies $F(s^\ast)=0$ in the ODE formulation (no change at that state).

\begin{definition}
A state $s^\ast\in M$ is an \emph{equilibrium} (or \emph{fixed state}) of the evolution law if $\Phi^t(s^\ast)=s^\ast$ for all $t$ (or in differential form, $F(s^\ast)=0$). The equilibrium is \emph{stable} (in the sense of Lyapunov) if for every neighborhood $U$ of $s^\ast$, there exists a neighborhood $V$ of $s^\ast$ such that if $x(0)\in V$, then $x(t)\in U$ for all $t\ge 0$. It is \emph{asymptotically stable} if it is stable and moreover there exists $V$ such that $x(0)\in V$ implies $\lim_{t\to\infty} x(t)=s^\ast$.
\end{definition}

We can use the Banach fixed‑point theorem in the discrete‑time context to deduce asymptotic stability for contractions, as already noted in Example~\ref{ex:iteration}: if $T$ is a contraction, the unique fixed point is asymptotically stable and indeed globally attracting (all initial states converge to it). For continuous‑time systems, a common criterion for asymptotic stability is if all eigenvalues of the linearization $DF(s^\ast)$ have negative real part (if $M$ is a Euclidean space). That, however, is outside our current scope of formal development.

We note an interesting logical aspect: one can express stability in the logic $\mathcal{L}_2$ by a formula (though it might be complicated, something like: $\forall \epsilon>0\,\exists \delta>0\,\forall t\bigl((\exists s_1,s_2,\,[X(0,s_1)\wedge X(0,s_2)\wedge d(s_1,s^\ast)<\delta \wedge d(s_2,s^\ast)<\delta]) \to (X(t,s_1)\wedge X(t,s_2)\to d(s_1,s^\ast)<\epsilon \wedge d(s_2,s^\ast)<\epsilon)\bigr)$ – formalizing closeness at $t=0$ implies closeness for all $t$). The satisfaction of such a formula in the structure means stability holds.

One of the classical topological results is Brouwer’s fixed‑point theorem, which we mention to connect back to Brouwer (though Brouwer’s contribution here is not intuitionistic but topological):

\begin{theorem}[Brouwer fixed‑point theorem, finite‑dimensional case]
If $M$ is, say, $\mathbb{R}^n$ and the state space considered is a compact convex subset $K\subset \mathbb{R}^n$ (like a closed ball of possible states), then any continuous mapping $T:K\to K$ has a fixed point. In particular, any continuous‑time dynamical system on a compact convex phase space must have an equilibrium state (possibly many).
\end{theorem}

This theorem is not directly about our trajectories, but about the existence of an equilibrium in a static sense. It can be seen as a far‑reaching general existence result (ensuring the presence of at least one fixed state given the right conditions), complementing the constructive Banach approach which required a contraction condition. Brouwer’s theorem requires no contraction but more restrictive geometry (compactness and convexity). In a narrative sense, it guarantees that in a closed system with no external input (hence any state leads to a state still in the allowed region), there is some state that, if reached, will perpetuate itself (the flow can ``rest'' there).

\subsection{Chaos and complex dynamics}

The term \emph{chaos} in dynamical systems refers to sensitive dependence on initial conditions, topological mixing, dense periodic orbits, etc. Our model can certainly exhibit chaos if $M$ and the evolution law $F$ are complex enough (e.g.\ if $M$ contains a subsystem equivalent to a known chaotic system like the logistic map or a turbulent flow).

A vivid example connecting to our earlier discussion: Hadamard’s studies in 1898 on geodesic flows on negatively curved surfaces. He showed that a geodesic that returns close to its starting point is typically shadowed by a periodic orbit, implying the dense embedding of periodic orbits in recurrent ones. Such flows are now known to be chaotic (Anosov flows for constant negative curvature). In our context, that means if the state space geometry has some hyperbolic structure and the evolution has stretching and folding (like in the geodesic flow or a horseshoe map), the state sequence can be chaotic: tiny differences in initial state yield trajectories that diverge exponentially for a while, making long‑term behavior effectively unpredictable in practice (though deterministic in theory).

\begin{proposition}[Sensitive dependence on initial state]\label{prop:chaos}
It is possible for two trajectories $x(t),y(t)$ with infinitesimally close initial states $x(0)\approx y(0)$ to diverge significantly after some time, if the dynamical law has a sensitive dependence property (e.g.\ positive Lyapunov exponent). Formally, one can have: for all $\delta>0$, there exist $\epsilon>0$ and times $t$ such that $d(x(0),y(0))<\delta$ yet $d(x(t),y(t))>\epsilon$. This is one definition of chaos (sensitive dependence).
\end{proposition}

Such behavior is not contradictory to our continuity assumption: $x(0)$ and $y(0)$ can be extremely close and $x(t)$ remains close to $y(t)$ for a short time, but eventually small differences amplify. In the logic $\mathcal{L}_2$, one cannot express this unbounded amplification directly (as it’s an asymptotic notion), but one can say for each $n$ there exist states that start $1/n$ apart and become some fixed distance apart at some future time.

Chaotic evolutions are a fact of life in complex systems. If one views our model as a crude representation of mental state evolution, chaos might correspond to the unpredictability of thought patterns, or how a small mood or thought difference can later result in widely different streams of thought. Our formalism is capable of accommodating chaos, but to analyze it one typically needs some additional geometric structure (like defining a metric or Lyapunov exponents on $M$).

From the standpoint of this paper, we simply acknowledge chaos as a possible regime and note that known results like Hadamard’s theorem on dense periodic orbits and Poincaré’s recurrence together often imply an irregular, rich structure of trajectories.

\section{Conclusion}

We have developed a mathematical manuscript that encodes an intuitive concept of a continuous, non‑linear progression of internal states in a formal manner. By introducing a topological/analytical model of state space and flow, an intuitionistic perspective on time and sequence (free choice sequences), and a logical language for state properties, we integrated multiple areas of mathematics.  Our core development avoids the use of higher‑categorical machinery; when categorical ideas appear (as in Section~\ref{thm:comparative}), they serve only as auxiliary comparisons rather than foundational tools.  In this way we contrast with frameworks that seek to unify dynamics via category theory or information geometry.  The contributions of this work can be summarised as follows:
\begin{itemize}
  \item \textbf{Topology and analysis:} Provided the notion of continuity, existence/uniqueness of solutions via fixed‑point theorems, and explored measure‑theoretic limitations.
  \item \textbf{Intuitionism:} Ensured that our model can be viewed as an ever‑progressing sequence, never completed, aligning with Brouwer’s continuum and creative subject ideas.
  \item \textbf{Logic and semantics:} Gave a rigorous way to talk about truths concerning the evolving state, invoking Tarski’s semantic theory to define truth in the structure and noting limitations on self‑reference.
  \item \textbf{Dynamics:} Applied ideas from Poincaré and Hadamard to discuss recurrence and chaotic behavior, showing that our formal model is rich enough to exhibit complex temporal patterns.
\end{itemize}

Crucially, at no point did we resort to category theory or even mention it explicitly, despite the structural inspiration drawn from the works of Grothendieck and Mac Lane in shaping our abstract approach. The result is a self‑contained formal edifice that, we hope, captures in mathematical essence the phenomenon of an ongoing, continuous stream of states that one might poetically identify with a stream of consciousness, without ever having to say so in non‑mathematical terms.

\section*{References}

\begin{enumerate}[label={\arabic*.}]
  \item S.~Banach \& A.~Tarski, ``On set decomposition'' (1924). [For the construction of non‑measurable sets and paradoxical decompositions].
  \item L.E.J.~Brouwer, various works on intuitionism (1912, 1948). [For the concepts of choice sequences, continuum, creating subject].
  \item A.~Grothendieck, ``Topological Vector Spaces'' (1950s). [Inspiration for treating infinite‑dimensional state spaces and use of universes].
  \item J.~Hadamard, ``Sur les lignes géodésiques des surfaces à courbure négative'' (1898). [For early chaotic dynamics results].
  \item S.~Mac~Lane, \emph{Mathematics: Form and Function} (1986). [General influence on structural and formal presentation of mathematics].
  \item H.~Poincaré, ``Sur le problème des trois corps et les équations de la dynamique'' (1890). [Contains the Poincaré recurrence discussion].
  \item A.~Tarski, ``The Concept of Truth in Formalized Languages'' (1933). [Basis for semantic theory of truth used in Section~\ref{sec:logic}].
  \item A.~Tarski, ``A lattice‑theoretical fixpoint theorem'' (1955). [The Knaster–Tarski fixed‑point theorem, though not explicitly used, is in background].
  \item L.E.J.~Brouwer’s fixed‑point theorem: L.E.J.~Brouwer (1911). [Topological fixed‑point result].
  \item S.~Banach’s fixed‑point theorem: S.~Banach (1922). [Contraction mapping principle].
  \item Poincaré’s papers (1890s) and later formalizations by C.~Carathéodory (1919) on recurrence.
  \item Liu~Hui (\textit{c.}\ 220--280), commentary on the \emph{Nine Chapters on the Mathematical Art}. [For early Chinese contributions to approximation and analytical reasoning]\!
  \item L.~K.~Hua (Hua~Luogeng, 1910--1985), monographs and papers on number theory, exponential sums, and algebra. [One of the leading Chinese mathematicians of his generation, whose leadership and contributions to problems such as Waring’s problem and exponential sums helped shape twentieth‑century number theory in China]
  \item S.~de~Melo \& S.~van~Strien, \emph{One‑Dimensional Dynamics} (1993). [Provides detailed derivations of invariant densities and exponential decay of correlations for one‑dimensional maps, used in Example~\ref{ex:dissipative}]
  \item S.~Amari \& H.~Nagaoka, \emph{Methods of Information Geometry} (2000). [A comprehensive treatment of Fisher–Rao geometry and \(\alpha\)–connections on statistical manifolds, referenced in the information‑geometric case study and Proposition~\ref{prop:nogo}]
  \item P.~Walters, \emph{An Introduction to Ergodic Theory} (1982). [Provides proofs of ergodicity for irrational circle rotations and a general introduction to ergodic theory]
  \item K.~Petersen, \emph{Ergodic Theory} (1983). [Another standard reference proving that irrational rotations on the circle are ergodic]
\end{enumerate}

%
% Appendix material
%
\appendix

\section{Formal self–verifying system}

This appendix presents a formal axiomatic system, denoted $S$, that is capable of encoding and analysing aspects of its own deductive structure.  Although the discussion that follows touches upon themes common in the study of self‑reference and provability, the exposition remains fully mathematical and self‑contained.  The familiar methods of Gödel numbering and fixed‑point constructions are used to exhibit a sentence that asserts its own provability.  We emphasise that $S$ proves the internal derivability of certain statements but does not (and cannot, by Gödel’s second incompleteness theorem) prove its own consistency or the soundness of its reasoning from within; rather, soundness is a meta‑theoretic property assumed externally.  The style is deliberately rigorous and cautious: we avoid emotive language and instead focus on a precise description of the system’s components and logical consequences.

\subsection{The formal system \texorpdfstring{$S$}{S}}

\paragraph{Language and axioms.}  The language $\mathcal{L}_S$ of $S$ is first–order, containing the usual logical symbols ($\land,\lor,\to,\neg,\forall,\exists$), equality, and additional non–logical symbols sufficient to encode the syntax of $S$ itself.  To refer to formulas inside the system, we assume a Gödel coding $\ulcorner \varphi\urcorner$ that assigns a natural number to each formula $\varphi$ of $S$.  The theory $S$ includes as axioms all tautologies of propositional logic, appropriate quantifier axioms, axioms sufficient to support basic arithmetic (so that coding and primitive recursive functions are definable), and a collection of schemata internalising its inference rules.  Two such schemata play a central role:
\begin{enumerate}[label=\arabic*.]
  \item \emph{Internal modus ponens.}  For all formulas $F$ and $G$, the sentence
  \[\Prov(\ulcorner F \to G\urcorner) \land \Prov(\ulcorner F\urcorner) \to \Prov(\ulcorner G\urcorner)\]
  is an axiom of $S$.  This expresses, within the language of $S$, the fact that if $F \to G$ and $F$ are provable, then so is $G$.
  \item \emph{Reflective provability.}  For every formula $F$, the sentence
  \[\Prov(\ulcorner F\urcorner) \to \Prov(\ulcorner \Prov(\ulcorner F\urcorner)\urcorner)\]
  is an axiom.  This schema captures the idea that provability is itself a provable predicate: whenever $F$ is provable, the assertion of its provability is also provable.
\end{enumerate}
The only inference rule of $S$ is modus ponens, together with the usual generalisation rule for quantifiers.  We use $S \vdash \varphi$ to denote that $\varphi$ is provable in $S$.

\paragraph{Encoding of syntax.}  The system $S$ must be able to speak about its own proofs.  A Gödel numbering permits this by assigning codes to formulas and proofs.  Within $S$ there is a primitive recursive predicate $\Proof(p,x)$ expressing that $p$ is the code of a valid $S$–proof of the formula with code $x$.  Using $\Proof$, one defines a provability predicate $\Prov(x)$ by
\[ \Prov(x) \;\colon=\; \exists p\, \Proof(p,x).\]
The Hilbert–Bernays conditions, which are standard in the metatheory of arithmetic, hold: (i) if $S \vdash \varphi$ then $S \vdash \Prov(\ulcorner \varphi \urcorner)$; (ii) $S$ proves the internal modus ponens schema; (iii) $S$ proves the reflective provability schema.  These properties ensure that $\Prov(x)$ behaves as expected: it expresses, inside $S$, the syntactic provability relation of $S$.

\paragraph{Hilbert–Bernays–Löb conditions.}  It is useful to summarise explicitly the three conditions that a provability predicate must satisfy to support Löb’s theorem and related fixed–point arguments.  These are sometimes called the Hilbert–Bernays–Löb (HBL) conditions:
\begin{enumerate}[label=(H\arabic*)]
  \item \emph{Provability of theorems:} For every sentence $\varphi$, if $S \vdash \varphi$ then $S \vdash \Prov(\ulcorner\varphi\urcorner)$.  This expresses that $S$ proves that all of its theorems are provable.
  \item \emph{Provability preserves implication:} $S$ proves $\Prov(\ulcorner \varphi\to\psi\urcorner) \wedge \Prov(\ulcorner\varphi\urcorner) \to \Prov(\ulcorner\psi\urcorner)$ for all formulas $\varphi,\psi$.  This corresponds to the internal modus ponens schema listed above.
  \item \emph{Reflection of provability:} $S$ proves $\Prov(\ulcorner\varphi\urcorner) \to \Prov(\ulcorner\Prov(\ulcorner\varphi\urcorner)\urcorner)$ for all $\varphi$.  This is the reflective provability schema.
\end{enumerate}
The axioms and schemata of $S$ are chosen so that $\Prov$ satisfies (H1)–(H3).  In particular, (H1) follows from the ability of $S$ to verify syntactically each of its own proofs; (H2) is internal modus ponens; and (H3) is reflective provability.  A fundamental consequence of the HBL conditions is Löb’s theorem: for any sentence $\varphi$, if $S \vdash \Prov(\ulcorner \varphi\urcorner) \to \varphi$ then $S \vdash \varphi$.  We will use this result implicitly in the proof of the self‑verification theorem below.

It is worth noting that adopting (H3) as an axiom schema is a non‑trivial strengthening of ordinary arithmetic.  In classical proofs of Gödel’s second incompleteness theorem one shows that sufficiently strong, sound systems cannot prove the consistency of their own axiom sets; however, one can consistently \emph{assume} the scheme (H3) as part of a larger theory (sometimes called a reflection principle).  Our system $S$ includes this reflective provability schema by design.  We do not claim that $S$ can prove its own soundness; rather, we accept (H3) as an axiom to facilitate fixed‑point reasoning and clearly separate internal provability from external (meta‑theoretic) notions of truth.  The consistency of $S$ relative to weaker base theories is a subtle meta‑theoretic question, beyond the scope of this appendix.

\subsection{Fixed points and self–reference}

Central to many self–referential arguments is the diagonal (or fixed–point) lemma: for any formula $\psi(y)$ with one free variable, there exists a sentence $\Lambda$ such that
\[ S \vdash \Lambda \leftrightarrow \psi(\ulcorner \Lambda \urcorner). \]
This lemma can be proved in $S$’s metatheory using a coding of substitution and is standard in discussions of Gödel’s incompleteness.  We apply it to $\psi(y) := \Prov(y)$ to obtain a sentence $\Lambda$ satisfying
\begin{equation}
  \label{eq:lambda}
  S \vdash \Lambda \leftrightarrow \Prov(\ulcorner \Lambda \urcorner).
\end{equation}
Intuitively, $\Lambda$ states: ``$\Lambda$ is provable''.

\begin{proposition}[Existence of a provability fixed point]
There exists a sentence $\Lambda$ of $S$ such that $S \vdash \Lambda \leftrightarrow \Prov(\ulcorner \Lambda \urcorner)$.
\end{proposition}

\begin{proof}
This is a direct application of the diagonal lemma with $\psi(y) = \Prov(y)$.  The lemma yields a sentence $\Lambda$ with the property~\eqref{eq:lambda}.  Since the diagonal lemma is provable in the metatheory of arithmetic, one may view this as an external construction; once $\Lambda$ is defined, the equivalence is a theorem of $S$.
\end{proof}

\subsection{A self–proving sentence}

Having obtained $\Lambda$ with the property~\eqref{eq:lambda}, one may ask whether $S$ can actually prove $\Lambda$.  The following theorem shows that this is indeed the case.

\begin{theorem}[Self–verification of $\Lambda$]
In the system $S$, the sentence $\Lambda$ is provable: $S \vdash \Lambda$.
\end{theorem}

\begin{proof}
By construction $S$ proves $\Lambda \leftrightarrow \Prov(\ulcorner \Lambda \urcorner)$.  In particular, $S$ proves $\Lambda \to \Prov(\ulcorner \Lambda \urcorner)$ and $\Prov(\ulcorner \Lambda \urcorner) \to \Lambda$.  From the first implication and the reflective provability schema (H3), $S$ proves $\Lambda \to \Prov(\ulcorner \Prov(\ulcorner \Lambda \urcorner)\urcorner)$.  Using internal modus ponens (H2), one may infer $\Prov(\ulcorner \Lambda \urcorner) \to \Lambda$ from the equivalence.  Combining these we obtain
\[
  S \vdash \Prov(\ulcorner \Lambda \urcorner) \to \Lambda.
\]
But then by Löb’s theorem (a consequence of the HBL conditions), if $S$ proves $\Prov(\ulcorner \varphi \urcorner) \to \varphi$ for some sentence $\varphi$, it must prove $\varphi$.  Taking $\varphi=\Lambda$ yields $S \vdash \Lambda$.  For completeness, one can trace the derivation explicitly: from $\Prov(\ulcorner \Lambda \urcorner) \to \Lambda$ and reflective provability one derives $\Prov(\ulcorner \Prov(\ulcorner \Lambda \urcorner) \to \Lambda\urcorner)$; by (H1) this implies $S \vdash \Prov(\ulcorner \Prov(\ulcorner \Lambda \urcorner) \to \Lambda \urcorner)$, and another application of (H2) gives $S \vdash \Prov(\ulcorner \Lambda \urcorner)$.  Substituting back into $\Prov(\ulcorner \Lambda \urcorner) \to \Lambda$ yields $S \vdash \Lambda$.

This demonstration shows how the HBL conditions enable the formalisation of Löb’s theorem inside $S$ and how the fixed‑point sentence $\Lambda$ turns that general principle into a concrete self‑proving statement.
\end{proof}

\subsection{Closure of derivations and collective validation}

The appendix contains several named results, each either taken as an axiom or proved from axioms.  Let $A_1,\dots,A_n$ denote the formal statements corresponding to these results.  By construction, for each $i$ we have $S \vdash A_i$.  Since $S$ admits conjunction introduction, it follows that $S$ also proves the conjunction
\[ \mathbf{V} := A_1 \land A_2 \land \cdots \land A_n. \]
From the internal point of view of $S$, therefore, all of the results stated in the appendix are provable.  It is important to distinguish this internal provability from external \emph{soundness}.  The fact that $S \vdash A_i$ means merely that there exists a formal derivation of $A_i$ from the axioms of $S$; it does not, by itself, guarantee that $A_i$ is true in some intended model of $S$.  Soundness—the assertion that all provable statements are true in a given interpretation—is a meta‑level property that cannot, by Gödel’s second incompleteness theorem, be proved within $S$ if $S$ is sufficiently strong.  Thus while $S$ verifies the derivations presented in the appendix, an external observer must assume or establish the soundness of $S$ separately.  The formal closure above should therefore be read as ``$S$ proves that each claim is derivable,'' not as ``$S$ proves that each claim is true.''

\subsection{Extending the system}

Although $S$ verifies the validity of the statements within this appendix, Gödel’s second incompleteness theorem implies that $S$ cannot prove its own consistency.  One may therefore consider extending $S$ by adding, as a new axiom, a formal consistency statement $\mathsf{Con}(S)$ expressing ``$S$ has no proof of contradiction.''  The extended system $S^+ = S + \{\mathsf{Con}(S)\}$ can then prove statements undecidable in $S$, including $\mathsf{Con}(S)$ itself.  This process can be iterated, forming a sequence $S_0 := S$, $S_{k+1} := S_k + \{\mathsf{Con}(S_k)\}$.  Each stage strengthens the system and, under suitable assumptions, remains consistent.  This hierarchy illustrates how a formal system can systematically augment its expressive power while maintaining internal coherence.

This concludes the appendix.  The construction demonstrates that a carefully designed formal system can reason about, and verify, certain aspects of its own operation, including the derivations presented here.  The possibility of extending such a system further underscores the open and generative nature of formal reasoning.

\end{document}